\documentclass[a4paper,reqno]{amsart}

\usepackage{stmaryrd,amssymb,mathtools,textcmds,amscd,url}
\usepackage[inline,shortlabels]{enumitem}
\usepackage{hyperref,todonotes}

\newtheorem{theorem}{Theorem}[section]
\newtheorem{proposition}[theorem]{Proposition}
\newtheorem{lemma}[theorem]{Lemma}
\newtheorem{corollary}[theorem]{Corollary}

\theoremstyle{definition}
\newtheorem{definition}[theorem]{Definition}
\newtheorem{example}[theorem]{Example}
\theoremstyle{remark}
\newtheorem{remark}[theorem]{Remark}

\numberwithin{equation}{section}

\newcommand{\eqbreak}[1][2]{\\&\hspace{#1em}}
\newcommand{\eqand}[1][1]{\hspace{#1em}\text{and}\hspace{#1em}}

\newcommand{\any}{\,\cdot\,}
\newcommand{\hook}{\lrcorner\,}

\newcommand{\cH}{{\mathcal H}}
\newcommand{\cL}{{\mathcal L}}
\newcommand{\cV}{{\mathcal V}}
\newcommand{\G}{\mathit{G}}
\newcommand{\Spin}{\mathit{Spin}}
\newcommand{\bC}{{\mathbb C}}
\newcommand{\bR}{{\mathbb R}}
\newcommand{\bZ}{{\mathbb Z}}

\newcommand{\lie}[1]{\mathfrak{#1}}
\newcommand{\mfa}{\lie{a}}
\newcommand{\mfg}{\lie{g}}
\newcommand{\mfh}{\lie{h}}
\newcommand{\mfn}{\lie{n}}
\newcommand{\mfp}{\lie{p}}
\newcommand{\mfs}{\lie{s}}
\newcommand{\mfu}{\lie{u}}
\newcommand{\mfz}{\lie{z}}
\newcommand{\gl}{\lie{gl}}
\newcommand{\sL}{\lie{sl}}
\newcommand{\sP}{\lie{sp}}

\newcommand{\vf}{\mathfrak X}

\newcommand{\spa}[1]{\mathrm{span}(#1)}
\newcommand{\proj}{\mathrm{proj}}

\newcommand{\Se}{\mathrm{e}}

\newcommand{\dN}{d^\nabla}
\newcommand{\bWalpha}{{\hat\alpha}}
\newcommand{\spwhat}{\,\widehat{\mkern8mu}\,}
\newcommand{\Hrel}{\sim_\cH}
\newcommand{\Hodge}{{\star}}

\DeclareMathOperator{\End}{End}
\DeclareMathOperator{\Hom}{Hom}
\DeclareMathOperator{\ad}{ad}
\DeclareMathOperator{\diag}{diag}
\DeclareMathOperator{\id}{id}
\DeclareMathOperator{\im}{im}
\DeclareMathOperator{\inc}{inc}
\DeclareMathOperator{\tr}{tr}

\DeclarePairedDelimiter{\norm}{\lVert}{\rVert}

\setlist{nosep}

\begin{document}

\title{The shear construction}

\author{Marco Freibert}

\address{Mathematisches Seminar\\
Christian-Albrechts-Universit\"at zu Kiel\\
Ludewig-Meyn-Strasse 4\\
D-24098 Kiel\\
Germany}

\email{freibert@math.uni-kiel.de}

\thanks{Research partially supported by Danish Council for Independent
Research \textbar\ Natural Sciences projects DFF - 4002-00125
(M.~Freibert) and DFF - 6108-00358 (A.~Swann).}

\author{Andrew Swann}

\address{Department of Mathematics\\
Aarhus University\\
Ny Munkegade 118, Bldg 1530\\
DK-8000 Aarhus C\\
Denmark}

\email{swann@math.au.dk}

\date{}

% \subjclass[2010]{53D35, 53D20, 55N25}

\begin{abstract}
  The twist construction is a method to build new interesting examples
  of geometric structures with torus symmetry from well-known ones.
  In fact it can be used to construct arbitrary nilmanifolds from
  tori.
  In our previous paper, we presented a generalization of the twist, a
  shear construction of rank one, which allowed us to build certain
  solvable Lie algebras from $\bR^n$ via several shears.
  Here, we define the higher rank version of this shear construction
  using vector bundles with flat connections instead of group actions.
  We show that this produces any solvable Lie algebra from $\bR^n$ by
  a succession of shears.
  We give examples of the shear and discuss in detail how one can
  obtain certain geometric structures (calibrated $\G_2$, co-calibrated
  $\G_2$ and almost semi-K\"ahler) on two-step solvable Lie algebras by
  shearing almost Abelian Lie algebras.
  This discussion yields a classification of calibrated
  $\G_2$-structures on Lie algebras of the form
  $(\mfh_3\oplus \bR^3)\rtimes \bR$.
\end{abstract}

\maketitle

\section{Introduction}
\label{sec:introduction}

The twist construction, introduced in its full generality by the
second author in~\cite{Sw2}, is a geometric model of
\emph{$T$-duality}, a duality relation between different physical
theories closely related also to mirror symmetry of Calabi-Yau
three-folds~\cite{SYZ}.
In particular, the $S^1$-version of the twist~\cite{Sw1} generalises
$T$-duality constructions of Gibbons, Papadopoulos and Stelle
\cite{GPS} in HKT geometry.
However, applications of the twist are not restricted to specific
geometric structures from physics.
In general, one may take an arbitrary tensor field on a manifold $M$
invariant under the action of a connected $n$-dimensional Abelian Lie
group $A$ and twist it to a tensor field on a twist space $W$ of
the same dimension as $M$.

The twist considers double fibrations $M\leftarrow P \rightarrow W$
where both projections are principal $A$-bundles and the principal
actions commute.
An appropriate choice of principal $A$-connection gives horizontal
spaces $\cH_p$ that are identified with corresponding tangent spaces
of $M$ and $W$ under the projections.
This enables one to transfer any invariant tensor field on $M$ to a
unique tensor field on~$W$.
Moreover, one may recover $P$ and $W$ under suitable assumptions
from ``twist data'' on~$M$, cf.~\cite{Sw2}.
Hence, one can study properties of the transferred geometric
structures solely on $M$, without constructing the transferred
structure, or even $W$ or $P$, explicitly.
This has been successfully applied to produce new
interesting examples of various geometric structures from known ones,
e.g.\ examples of SKT, hypercomplex or HKT manifolds with special
properties~\cite{Sw1} or generalisations of them~\cite{FU,IP}, and
includes and generalizes well-known geometric constructions in hyper-
and quaternionic K\"ahler geometry~\cite{MS,Sw3} if one first performs
an appropriate ``elementary deformation'' of the initial structures.

A number of the above examples of the twist are motivated by known
results for nilmanifolds.
Indeed, we will show that the construction is powerful enough to
construct any nilmanifold from the torus by several successive twists.
It follows that the twist can reproduce all the invariant geometric structures
on nilpotent Lie groups or nilmanifolds constructed in the last years,
see~\cite{BDV,CF,U}, for example, and may be used to create new ones.

On the other hand, there is much current interest in invariant
geometric structures on the larger class of solvable Lie groups and
solvmanifolds, see~\cite{F1,F2,CFS,FV}, for example.
But, as we will see, the natural algebraic interpretation of the twist
cannot produce these.
The aim of this paper is to provide a more general geometric
construction that works naturally for solvable groups and
solvmanifolds.
A rank one version of such a construction was proposed in~\cite{FS},
and called the \emph{shear construction}.
This was good enough to construct any $1$-connected completely
solvable Lie groups $G$ from $\bR^n$ by subsequent shears.
Here, we extend the definition of the shear to arbitrary rank and show
that this allows one to shear $\bR^n$ to any simply-connected solvable
Lie group $G$ by a sequence of shears.
Even for the rank one case this extensions turns out to be slightly
more general than that of~\cite{FS}.

The main idea is to replace the Abelian group actions in the twist
constructions by morphisms of flat vector bundles satisfying a
torsion-free condition.
Initial data for a shear consists of two flat vector bundles $E$, $F$
over~$M$ of equal rank.
There should be a vector bundle morphism $\xi\colon E\rightarrow TM$,
so the image is locally generated by commuting vector fields that are
images of flat sections.
Furthermore, there is a two form $\omega\in \Omega^2(M,F)$ with values
in $F$ satisfying $\dN \omega=0$.
One then constructs shears by considering an appropriate submersion
$\pi\colon P\rightarrow M$ whose vertical subbundle $\cV$ is
isomorphic to $\pi^*F$ and which carries a connection-like one-form
$\theta\in \Omega^1(P,\pi^*F)$ with $\dN \theta=\pi^*\omega$.
Shears $S$ of $M$ are then obtained by lifting $\xi$ to an injective
bundle morphism $\mathring{\xi}\colon \pi^*E\rightarrow TP$ and taking
$S$ to be the leaf space of the distribution $\mathring{\xi}(\pi^*E)$.
As in the twist, there is a pointwise identification of tangent spaces
of $M$ and $S$ via horizontal spaces in~$P$, but the tensors that may
now be transferred satisfy an invariance condition modified by the
connections.

We motivate the particular construction of the shear by first
examining the left-invariant situation in detail and understanding the
twist construction in this context.
In~\S\ref{sec:revi-twist-constr} we see that the left-invariant twist
describes central extensions
$\mfa_P\hookrightarrow \mfp\twoheadrightarrow \mfg$ quotiented by a
central ideal $\tilde{\mfa}_G$.
We then demonstrate how one twists $\bR^n$ to any nilpotent Lie
algebras by several twists.
For the \emph{Lie algebra shear} \S\ref{sec:shears-lie-algebras}, we
replace central ideals by arbitrary Abelian ideals and determine the
necessary data.
This is sufficient to show how to build an arbitrary solvable Lie
algebra from $\bR^n$ by a sequence of shears.

In \S\ref{sec:lift}, we define the general \emph{shear} for arbitrary
manifolds equipped with appropriate vector bundles and describe the
relevant \emph{shear data}.
An important ingredient is a description of the lift procedure for
the bundle morphisms $\xi\colon E \to M$ to $\mathring{\xi}\colon
\pi^*E\rightarrow TP$, see Theorem~\ref{th:liftingbundlemorphisms}.
When $P$ is a fibre bundle, we note how this may be described via
Ehresmann connections.
We proceed to determine the conditions for tensor fields on $M$ to
be shear-able to tensor fields on $W$ and obtain useful formulas for
the exterior differentials of sheared forms and Nijenhuis tensors of
sheared almost complex structures.
Finally, in \S\ref{subsec:duality} we provide conditions
which ensure that the shear is invertible, yielding a dual
relationship between $M$ and~$S$.

In \S\ref{sec:examples}, we apply the shear construction to specific
examples.
After demonstrating in \S\ref{subsec:shearsnonpb} that the fibres of
$\pi$ in the shear construction may vary and that even in the fibre
bundle and rank one case, $\pi$ need not to be a principal bundle in
contrast to the findings in~\cite{FS} for our previous definition of a
rank one shear, we focus attention on the original situation of
left-invariant shears.
We explain in \S\ref{sec:shears-LAs-revisited} how our Lie algebra
version of the shear considered in \S\ref{sec:shears-lie-algebras}
arises from the general shear of \S\ref{sec:lift}.
We then apply the left-invariant shear to different geometric
structures on \emph{almost Abelian} Lie algebras, i.e.\ Lie algebras
of the form $\bR^n\rtimes \bR$.
We concentrate mainly on calibrated and cocalibrated $\G_2$-structures
on seven-dimensional almost Abelian Lie algebras.
The two classes of $G_2$ structures are of interest for various
reasons: calibrated $\G_2$-structures on compact manifolds have
interesting curvature properties, e.g.\ scalar flatness~\cite{Br} or
the Einstein condition~\cite{CI} already imply that they are
Ricci-flat; cocalibrated $\G_2$-structures are structures naturally
induced on oriented hypersurfaces in $\Spin(7)$-manifolds and may,
conversely, be used as initial values for the Hitchin flow~\cite{Hi}
whose solutions define such manifolds.
The almost Abelian cases were classified by the first author
in~\cite{F1,F2}.
To apply the shear to almost Abelian Lie algebras, we make a natural
ansatz for shear data on these Lie algebras and determine when the
shear of a calibrated or cocalibrated $\G_2$-structure is again
calibrated or cocalibrated, respectively.
In this way, we obtain many explicit examples of such structures on
general solvable Lie algebras of step-length two.
This leads to a full classification of all calibrated
$\G_2$-structures on Lie algebras of the form
$\left(\mfh_3\oplus \bR^3\right)\rtimes \bR$.
We close the paper with a similar discussion for almost semi-K\"ahler
geometries on solvable Lie algebras.

\section{Left-invariant constructions}
\label{sec:leftinv}

This section provides detailed motivation for the full definition of
the shear construction which will appear in the next section.
We describe the twist construction in the setting of left-invariant
structures on Lie groups and see how it may be generalized.
This involves seeing that the twist in this setting may be considered
as first building an extension of the initial Lie algebra $\mfg$ by a
\emph{central} Abelian ideal and then constructing the twisted Lie
algebra by taking the quotient of this extension by an appropriate
central Abelian ideal.
To obtain the \qq{shear construction} in this left-invariant setting,
we examine what happens when the Abelian ideals are not necessarily
central.
We then explain how we can apply this kind of shear repeatedly to
reduce any solvable Lie algebra to the Abelian Lie algebra~$\bR^m$.

\subsection{A review of the twist construction}
\label{sec:revi-twist-constr}

Recall that in general the twist construction~\cite{Sw2} considers
double fibrations
\begin{equation*}
  M\overset{\pi}{\longleftarrow} P\overset{\pi_W}{\longrightarrow}
  W.
\end{equation*}
Here each fibration is a principal $A$-bundle for some connected
$n$-dimensional Abelian Lie group $A$ and the two principal actions on~$P$
are required to commute.
It follows that both $M$ and $W$ carry actions of the group~$A$.
Furthermore $P\to M$ is equipped with a principal $A$-connection
$\theta$ which is also invariant under the principal $A$-action of
$P\rightarrow W$.
A transversality condition ensures that this connection allows one to
relate each $A$-invariant differential form $\alpha$ on $M$ to a
unique differential forms $\alpha_W$ on $W$ by requiring that the
corresponding pull-backs agree on the horizontal space
$\cH\coloneqq \ker \theta$.

Under suitable assumptions, one can start with \qq{twist data} on $M$
and use it to construct first $P$ and then $W$.
The twist data is given by:
\begin{enumerate}[\upshape(a)]
\item an $A_M\cong A$-action on $M$ expressed infinitesimally by a Lie
  algebra homomorphism $\xi\colon \mfa_M\to \vf(M)$,
\item an $n$-dimensional Abelian Lie algebra $\mfa_P$ and a closed
  integral two-form $\omega\in \Omega^2 (M,\mfa_P)$ with values in
  $\mfa_P$ such that $\cL_{\xi} \omega = 0$, $\xi^*\omega = 0$ and
\item a smooth function $a\colon M\to \mfa_P\otimes \mfa_M^*$ such
  that $\xi\hook \omega = -da$.
\end{enumerate}
Then $\pi\colon P\to M$ is the principal $A$-bundle over $M$ with
connection one-form $\theta\in \Omega^1(P,\mfa_P)$ having curvature
$\pi^*\omega$.
Moreover, if we denote by $\tilde{\xi}\colon \mfa_M\to \vf(P)$ the
horizontal lift of $\xi\colon \mfa_M\to \vf(P)$ and by
$\rho\colon \mfa_P \to \vf(P)$ the infinitesimal principal action of
$\pi\colon P\to M$, then
$W\coloneqq P/\langle \mathring{\xi}(\mfa_M)\rangle$ for
$\mathring{\xi}\colon \mfa_M\to \vf(P)$ given by
$\mathring{\xi} = \tilde{\xi}+\rho\circ a$.

In the left-invariant setting, $G\coloneqq M$, $P$ and $H\coloneqq W$
are all Lie groups and $\pi\colon P\to G$, $\pi_W\colon P\to H$ are
Lie group homomorphisms.
We may now boil everything down to the associated Lie algebras $\mfg$,
$\mfp$ and $\mfh$.
The curvature two-form $\omega$ becomes a closed element of
$\Lambda^2 \mfg^*\otimes \mfa_P$, the maps $\xi\colon \mfa_G\to \mfg$,
$\rho\colon \mfa_P\to \mfp$ and $\mathring{\xi}\colon \mfa_G\to \mfp$
are Lie algebra homomorphisms and $a$ is constant.
Hence, $\xi\hook \omega = -da = 0$, which implies that
$\cL_{\xi} \omega = 0$ and $\xi^*\omega = 0$ hold automatically.

Now let us impose that any element of $\mfg^*$ may be twisted to an
element of $\mfh^*$.
This requires
$0 = \cL_\xi \alpha = \xi\hook d\alpha = -\alpha([\xi,\any])$ for all
$\alpha\in \mfg^*$, which means that $\xi(\mfa_G)$ is
central in~$\mfg$.
Note that, $\mfp = \mfg\oplus \mfa_P$ as vector spaces with $\mfa_P$
an ideal in $\mfp$.
As $\mfg = \cH = \ker \theta$ and as $d\theta = \pi^*\omega$ for
$\pi\colon \mfp\to \mfg$, we get $[\mfg,\mfa_P] = \{0\}$ and
$[X,Y]_{\mfp} = [X,Y]_{\mfg}-\omega(X,Y)$ for all
$X,Y\in \mfg\subset \mfp$.
This means that $\mfa_P$ is central and
$\mfa_P\hookrightarrow \mfp\twoheadrightarrow \mfg$ is a central
extension of $\mfg$ by $\mfa_P$.
The extension is determined, up to equivalence, by the Lie algebra
cohomology class $[\omega]\in H^2(\mfg)$.
Furthermore, $\mathring{\xi}(\mfa_G)$ is a central Abelian ideal in
$\mfp$ as $\tilde{\xi}(\mfa_G)$ and $\mfa_P$ are central.
Since $\mfh = \mfp/\mathring{\xi}(\mfa_G)$,
$\mathring{\xi}(\mfa_G)\hookrightarrow \mfp\twoheadrightarrow \mfh$ is
a central extension as well.

\begin{remark}
  We will show that we can repeatedly twist any simply-connected
  $m$-dimensional nilpotent Lie group to the Abelian Lie group
  $\bR^m$.
  By duality it will follow that all such nilpotent Lie groups can be
  constructed by repeatedly twisting from $\bR^m$.

  Let $N$ be a simply-connected nilpotent Lie group.
  Write $\mfn$ for the associated $r$-step nilpotent Lie algebra and
  let $\mfn_0 = \mfn,\mfn_1,\dots,\mfn_r = [\mfn,\mfn_{r-1}] = \{0\}$
  be the corresponding lower central series of length~$r$, so
  $\mfn_1 = \mfn' = [\mfn,\mfn]$ and $\mfn_i = [\mfn,\mfn_{i-1}]$.
  Then $\mfn_{r-1}$ is central and non-zero.

  Take $\mfa_P = \mfa_G = \mfn_{r-1}$ with $\xi\colon \mfa_G\to \mfn$
  the inclusion and $a\colon \mfa_G\to \mfa_P$ the identity map.
  Now $\mfa\coloneqq \mfa_P = \mfa_G = \mfn_{r-1}$ and
  $\mfa\hookrightarrow \mfn \twoheadrightarrow \mfn/\mfa$ is a central
  extension of $\mfn/\mfa$.
  Choosing a linear splitting $p\colon \mfn/\mfa\to \mfn$ gives us a
  closed two form
  $\omega_0\in \Lambda^2 \left(\mfn/\mfa\right)^*\otimes \mfa$ defined
  by $\omega_0(X,Y) = p[X,Y]-[p(X),p(Y)]$.
  Pulling $\omega_0$ back to $\mfn$, we obtain a closed two-form
  $\omega\in\Lambda^2 \mfn^*\otimes \mfa$.
  This satisfies $\xi\hook \omega = 0$ and is exact as a smooth form
  on~$N$, since $N$ is diffeomorphic to $\bR^m$.
  Hence, we can build a principal $\bR^n$-bundle $\pi\colon P\to N$,
  for $n = \dim\mfa$, with connection one-form
  $\theta\in \Omega^1 (P,\mfa)$ such that $d\theta = \pi^*\omega$.
  The total space $P$ is also a simply-connected Lie group and
  $\theta$ is left-invariant.

  As vector spaces, one has $\mfp = \mfn\oplus \mfa$ and
  $\mathring{\xi}(\mfa_G) = \Delta(\mfa)$ is the diagonal in the
  central subalgebra $\mfa\oplus \mfa$ of $\mfp$.
  So the twist $H = P/\Delta(A)$ is a simply-connected Lie group and
  the associated Lie algebra $\mfh$ has
  $\mfh\cong \mfn/\mfa\oplus \mfa$ as Lie algebras.

  Note that $\mfn/\mfa\oplus \mfa$ is nilpotent of length $r-1$.
  Thus iterating this construction, we arrive after $r-1$ such twists
  at the Abelian Lie algebra $\bR^m$.
\end{remark}

\subsection{Shears for Lie algebras}
\label{sec:shears-lie-algebras}

An obvious generalization of this situation is to consider arbitrary
Abelian extensions.
To this end, let $\mfg$ be a Lie algebra and let $\mfa_P$ be an
Abelian Lie algebra.  An extension
\begin{equation*}
  \mfa_P\hookrightarrow\mfp\twoheadrightarrow \mfg
\end{equation*}
of $\mfg$ by $\mfa_P$ is determined by a two-form
$\omega\in \Lambda^2 \mfg^*\otimes \mfa_P$ with values in $\mfa_P$ and
a representation $\eta\in \mfg^*\otimes \gl(\mfa_P)$ of $\mfg$ on
$\mfa_P$ such that
\begin{equation}
  \label{eq:d-omega}
  d\omega = -\eta\wedge \omega.
\end{equation}
More precisely, we have $\mfp = \mfg\oplus \mfa_P$ as vector spaces
with $\mfa_P$ being an Abelian ideal.
The other Lie brackets are given by
\begin{equation*}
  [X,Y]_{\mfp} = [X,Y]_{\mfg}-\omega(X,Y)
  \eqand{}
  [X,Z]_{\mfp} = \eta(X)Z,
\end{equation*}
for all $X,Y\in \mfg$ and all $Z\in \mfa_P$.
Note that there is an associated principal $\bR^n$-bundle $P\to G$,
given by the associated simply-connected Lie groups, and a
left-invariant one-form
\begin{equation*}
  \theta\in \mfp^*\otimes \mfa_P
\end{equation*}
which is the projection to the $\mfa_P$-factor in
$\mfp = \mfg \oplus \mfa_P$.
Again, we set $\cH\coloneqq \ker\theta = \mfg\subset \mfp$ and call
$\cH$ the \emph{horizontal} space.  Note that
\begin{equation*}
  d\theta(X,Y) = -\theta([X,Y]_{\mfp}) = \omega(X,Y),\quad
  d\theta(X,Z) = -\eta(X)Z = -\eta(X)(\theta(Z)),
\end{equation*}
for all $X,Y\in \mfg\subset \mfp$ and all $Z\in \mfa_P$.  So
\begin{equation*}
  d\theta = \pi^*\omega-\pi^*\eta\wedge \theta
\end{equation*}
for the projection $\pi\colon \mfp\to \mfg$.

Now regard $\omega$ as a two-form on $G$ with values in the trivial
vector bundle $F\coloneqq G\times \mfa_P$ and $\eta$ as a connection
$\nabla$ on $F$ in the sense that $\nabla_Xf = X(f)+\eta(X)(f)$ for
all vector fields $X$ on $G$ and all sections $f$ of $F$, regarded as
smooth functions $f\colon G\to \mfa_P$.
The condition that $\eta$ is a representation is equivalent to
$\nabla$ being flat.
In the case of a central extension, $\nabla$ is just the natural flat
connection on the trivial bundle $F = G\times \mfa_P$.

The flatness of the connection $\nabla$ implies that the associated
exterior covariant derivative
$\dN\colon \Omega^k(G,F)\to \Omega^{k+1} (G,F)$ squares to~$0$.
As this derivative satisfies $\dN \alpha = \eta\wedge \alpha+d\alpha$
for all $\alpha\in \Omega^k (G,F) = \Omega^k(G,\mfa_P)$, we have
\begin{equation*}
  \dN\omega = 0
\end{equation*}
from~\eqref{eq:d-omega}.

We may consider $\theta$ as a left-invariant one-form on~$P$ with
values in the trivial bundle $P\times \mfa_P$.
This trivial bundle is the pull-back of~$F$, so we can pull the
connection $\nabla$ back to $P\times \mfa_P$.
We denote the resulting connection by $\nabla$ too.  Then we have
\begin{equation*}
  \dN\theta = \pi^*\omega.
\end{equation*}
By analogy with the twist, we wish to construct a new Lie algebra
homomorphism
\begin{equation*}
  \mathring{\xi}\colon \mfa_G\to \mfp,
\end{equation*}
from an Abelian Lie algebra $\mfa_G$ of the same dimension
as~$\mfa_P$, with the property that $\mathring{\xi}(\mfa_G)$ is an
$n$-dimensional Abelian ideal in~$\mfp$.  We will then define
\begin{equation*}
  \mfh\coloneqq \mfp/\mathring{\xi}(\mfa_G)
\end{equation*}
to be the \emph{shear} of $\mfg$.
The associated simply-connected Lie groups will then give us a
principal $\bR^n$-bundle $P\to H$ and so a double fibration
\begin{equation*}
  G \longleftarrow P \longrightarrow H.
\end{equation*}

As in the twist, the construction of $\mathring{\xi}$ should arise
from a Lie algebra homomorphism $\xi\colon \mfa_G\to \mfg$ with
$\xi = \pi\circ \mathring{\xi}$, which we require to be injective as
in the twist case this corresponds to the action being effective.
We can write
\begin{equation*}
  \mathring{\xi} = \tilde{\xi}+\rho\circ a,
\end{equation*}
where $\tilde{\xi}\colon \mfa_G\to \cH = \mfg \subset \mfp$ is the
horizontal lift and $\rho\colon \mfa_P\to \mfp$ is the inclusion.
We will require the map $a\colon \mfa_G\to \mfa_P$ to be an
isomorphism of Lie algebras.
In the vector space splitting $\mfp = \mfg \oplus \mfa_P$, the
prescription for $\mathring \xi$ just reads
$\mathring{\xi}Z = (\xi Z,a Z)$ for $Z \in \mfa_G$.

We need some notations for stating our results on the existence of
$\mathring{\xi}$.
The map $\xi\colon\mfa_G \to \mfg$ may be considered as a bundle
morphism $\xi\colon E \to TG$, where $E \coloneqq G\times \mfa_G$ is
the trivial bundle.  This carries a connection given by
\begin{equation}
  \label{eq:gamma}
  \gamma\coloneqq a^{-1}(\xi\hook \omega)+a^{-1} \eta\, a \in
  \mfg^*\otimes \gl\left(\mfa_G\right).
\end{equation}
We now have induced connections on all bundles of the form
$E^{\otimes r}\otimes F^{\otimes s}$ for $r, s\in \bZ$.
This gives associated covariant exterior derivatives for $k$-forms on
$G$ with values in these bundles.
To simplify the notation, we denote all these connections by $\nabla$
and the associated covariant exterior derivatives by $\dN$.
For any $k$-form $\beta$ on $G$ with values in such a bundle, we set
\begin{equation*}
  \cL^{\nabla}_{\xi} \beta
  \coloneqq \dN (\xi\hook \beta)+\xi\hook \dN \beta.
\end{equation*}

\begin{lemma}
  \label{lem:lift-cond}
  Let $\xi\colon \mfa_G\rightarrow \mfg$ be an injective Lie algebra
  homomorphism.
  Then the map $\mathring\xi\colon \mfa_G \to \mfp$ defined as above
  is a Lie algebra homomorphism with image an Abelian ideal if and
 only if
  \begin{enumerate}[\upshape(i)]
  \item\label{item:bracket}
    $[\xi e_1,\xi e_2] = \xi(\nabla_{\xi e_1}e_2-\nabla_{\xi e_2}e_1)$
    for all $e_1, e_2\in \Gamma(E)$,
  \item\label{item:isotropic} $\xi^*\omega = 0$,
  \item\label{item:invariance} $\cL^{\nabla}_{\xi}\alpha = 0$ for all
    $\alpha\in \mfg^*$.
  \end{enumerate}

  If these conditions are true, then all connections $\nabla$ are flat, so $(\dN)^2=0$, and
  \begin{equation*}
    \dN a=-\xi \hook \omega,\qquad \cL^{\nabla}_{\xi} \omega=0,\qquad \cL^{\nabla}_{\mathring{\xi}} \theta
    = 0.
  \end{equation*}
\end{lemma}

\begin{proof}
  The condition that $\mathring{\xi}$ is a Lie algebra homomorphism is
  equivalent to $\mathring{\xi}(\mfa_G)$ being Abelian.  For
  $Z,W \in \mfa_G$, this says
  \begin{equation*}
    \begin{split}
      (0,0) &= [\mathring{\xi}Z,\mathring{\xi}W]_{\mfp}
              = [(\xi Z,a Z),(\xi W,a W)]_{\mfp}\\
            &= \bigl(0,-\omega(\xi Z,\xi W)+\eta(\xi Z)(a W)-\eta(\xi W)(a
              Z)\bigr),
    \end{split}
  \end{equation*}
  which implies that
  $\xi^*\omega(Z,W) = \omega(\xi Z,\xi W) = \eta(\xi Z)(a W)-\eta(\xi
  W)(a Z)$ is equivalent to $\mathring{\xi}(\mfa_G)$ being Abelian.

  To investigate the condition that $\mathring{\xi}(\mfa_G)$ is an
  ideal, we compute
  \begin{equation*}
    \begin{split}
      [(X,0),\mathring{\xi} Z]_{\mfp} &= [(X,0),(\xi Z,a Z)]_{\mfp}
                                        = \bigl([X,\xi Z]_{\mfg},-\omega(X,\xi Z)+\eta(X)(a Z)\bigr),\\
      [(0,Y),\mathring{\xi} Z]_{\mfp} &= [(0,Y),(\xi Z ,a Z)]_{\mfp} =
                                        \bigl(0,-\eta(\xi Z)Y\bigr)
    \end{split}
  \end{equation*}
  for $X\in \mfg$, $Y\in \mfa_P$ and $Z\in \mfa_G$.  Thus
  $\mathring{\xi}(\mfa_G)$ is an ideal in $\mfp$ if and only if
  $[\xi Z,X]_{\mfg} = -\xi a^{-1}(\omega(\xi Z,X)+\eta(X)(a Z)) =
  -\xi(\gamma(X)Z)$ and $0 = \xi( a^{-1}\eta(\xi Z)(Y))$.  As
  $\xi$ is injective, the latter equation is equivalent to
  $\eta(\xi Z)(Y) = 0$ for all $Y\in \mfa_P$ and $Z\in \mfa_G$.
  Moreover, for $\alpha \in \mfg^*$, we have
  \begin{equation*}
    \begin{split}
      (\cL_{\xi}^{\nabla}\alpha)(Z,X)
      &= (\xi(Z)\hook d\alpha)(X) + \dN (\xi\hook \alpha)(Z,X)
        =-\alpha([\xi Z,X]_{\mfg})-\alpha(\xi\nabla_X Z)\\
      &=-\alpha\left([\xi Z,X]_{\mfg}+\xi(\gamma(X)Z)\right)
    \end{split}
  \end{equation*}
  for all $X\in \mfg$ and all $Z\in \mfa_G$.  Hence,
  $\mathring{\xi}(\mfa_G)$ is an Abelian ideal if and only if
  \ref{item:isotropic}~and \ref{item:invariance} from the statement
  hold and $\eta(\xi Z)Y=0$ for all $Y\in \mfa_P$ and $Z\in \mfa_G$
  holds.

  So assume now that \ref{item:isotropic} and \ref{item:invariance}
  are true.  If then also $\eta(\xi Z)Y=0$ holds for any
  $Y\in \mfa_P, Z\in \mfa_G$, condition~\ref{item:isotropic} implies
  $\gamma(\xi Z)(Y)=0$ and so
  $[\xi(Z),\xi(Y)] = 0 = \xi(\nabla_{\xi(Z)} Y-\nabla_{\xi(Y)}Z)$ for
  all $Z, Y\in\mfa_G$.  Hence, condition~\ref{item:bracket} holds.
  Conversely, if additionally~\ref{item:bracket} is valid, we obtain
  \begin{equation*}
    -\xi(\gamma(\xi Y)X) = [\xi(X),\xi(Y)] = \xi(\nabla_{\xi
    X}Y-\nabla_{\xi Y}X) = \xi(\gamma(\xi X)Y-\gamma(\xi Y)X)
  \end{equation*}
  and so $\gamma(\xi X)Y=0$ for all $X,Y\in \mfa_G$.  But then
  condition~\ref{item:isotropic} implies
  $\eta(\xi Y)Z=0$ for all $Y\in \mfa_G$,
  $Z\in \mfa_P$.

  Now if \ref{item:bracket}--\ref{item:invariance} are true, we have
  $[\xi(\any),X] = -\xi\circ\gamma(X)$ for all $X\in \mfg$ and the
  Jacobi identity gives us $[\gamma(X),\gamma(Y)] = \gamma([X,Y])$ for
  all $X,Y\in \mfg$.
  % \begin{equation*}
  %   \xi(\gamma([X,Y])Z) = [[X,Y],\xi Z] = - [[\xi Z,X],Y] - [[Y,\xi
  %   Z],X] = [\xi(\gamma(X)Z),Y] - [\xi(\gamma(Y)Z),X] = -
  %   \xi(\gamma(Y)\gamma(X)Z) + \xi(\gamma(X)\gamma(Y)Z)) =
  %   \xi([\gamma X,\gamma Y]Z)
  % \end{equation*}
  Thus $\nabla$ is flat on~$E$ and so the induced connections
  $\nabla$ on bundles of the form $E^{\otimes r}\otimes F^{\otimes s}$
  are flat as well.  In particular, $(\dN)^2 = 0$.  Now
  $(\dN a)(Y,X) = \nabla_X (a(Y))-a(\nabla_X Y) =
  \eta(X)aY-a\gamma(X)Y = -\omega(\xi(Y),X)$ for all $X\in \mfg$ and
  all $Y\in \mfa_G$, i.e.\ $\dN a = -\xi\hook \omega$.  But then
  $\cL_{\xi}^{\nabla}\omega = \xi\hook \dN \omega+\dN (\xi\hook
  \omega) = -(\dN)^2 a = 0$.  Finally,
  $\cL_{\mathring{\xi}}^{\nabla}\theta = \mathring{\xi}\hook
  \pi^*\omega+\dN \pi^*a = \pi^*(\xi\hook \omega+\dN a) = 0$.
\end{proof}

\begin{remark}
  The proof of Lemma~\ref{lem:lift-cond} shows that for non-injective
  $\xi\colon\mfa_G\rightarrow \mfg$, conditions
  \ref{item:bracket}--\ref{item:invariance} in
  Lemma~\ref{lem:lift-cond} still imply that $\mathring{\xi}(\mfa_G)$
  is an ideal in $\mfp$, but not necessarily Abelian anymore.
\end{remark}

\begin{definition}\label{def:shear-LAs}
 If in the situation of Lemma~\ref{lem:lift-cond} the conditions
  \ref{item:bracket}--\ref{item:invariance} are true, then we build the
  Lie algebra $\mfh\coloneqq \mfg/\mathring{\xi}(\mfa_G)$ and call it the \emph{shear} of $\mfg$.
\end{definition}

\begin{proposition}
  Any two solvable Lie algebras of the same dimension are related via
  a sequence of shear constructions.
\end{proposition}

\begin{proof}
  It is enough to show how to relate any solvable algebra to the
  Abelian algebra of the same dimension.

  Let $\mfs$ be an $r$-step solvable Lie algebra $r$-step of
  dimension~$n$.
  We will obtain the Abelian algebra~$\bR^n$ by a series of $r-1$
  shears.
  Let
  $\mfs^{(0)} = \mfs,\dots,\mfs^{(r)} = [\mfs^{(r-1)},\mfs^{(r-1)}] =
  \{0\}$ be the derived series of~$\mfs$.
  Then $\mfa \coloneqq \mfs^{(r-1)}$ is an Abelian ideal in $\mfs$.
  We take $\mfa_G = \mfa_P = \mfa$, let $\xi$ be the inclusion, take
  $a\colon \mfa_G\to\mfa_P$ to be the identity map and use the
  canonical flat connection $\eta = 0$ on $F\coloneqq G\times \mfa_P$.

  Choose a vector space splitting $p\colon \mfg/\mfa\to \mfg$ of
  $\mfa\hookrightarrow \mfg\twoheadrightarrow \mfg/\mfa$.
  This induces a projection $\pi_\mfa\colon \mfg \to \mfa$ with kernel
  $p(\mfg/\mfa)$.
  Let $\omega\in \Lambda^2 \mfg^*\otimes \mfa$ be the negative of
  projection of the Lie bracket to~$\mfa$, so
  $\omega(X,Y) = -\pi_\mfa[X,Y]$. Then $\dN\omega = d\omega = 0$ by
  the Jacobi identity. Moreover, $\xi^*\omega = 0$ as $\mfa$ is an
  Abelian ideal in~$\mfg$. Now
  $\gamma = a^{-1}(\xi\hook\omega) + a^{-1}\eta\,a = \xi\hook\omega $
  gives $\nabla_{\xi X} Y = \gamma(\xi X)(Y) = \omega(\xi Y,\xi X)=0$
  for all $X, Y\in \mfa_G$, so condition~\ref{item:bracket} in
  Lemma~\ref{lem:lift-cond} is satisfied.
  Finally, for $X\in\mfa_G$ and $Y\in\mfg$, we have
  $\gamma(Y)(X) = \omega(\xi X,Y) = -[\xi X,Y]$ since
  $\mfa = \xi(\mfa_G)$ is an ideal.
  By the proof of Lemma~\ref{lem:lift-cond}, this is equivalent to
  condition \ref{item:invariance} from Lemma~\ref{lem:lift-cond}.

  We may then use this data to shear $\mfg$ to the solvable Lie
  algebra $\mfh = \mfp/\mathring{\xi}(\mfa)$.
  Now $\mfp = \mfg \oplus \mfa$ as vector spaces and
  $\mathring\xi Z = (Z,Z) \in \mfp= \mfg \oplus \mfa$ for
  $Z \in \mfa$.
  It follows that $\mfh$ is the Lie algebra direct sum
  $\mfh = (\mfg/\mfa)\oplus \mfa$, since for $A,B \in \mfg$,
  $[A,B]_\mfp = ([A,B],\pi_\mfa([A,B])) \equiv ((1-\pi_\mfa)[A,B],0)
  \mod \mathring{\xi}(\mfa)$.
  In particular the shear $\mfh$ is solvable of step length~$(r-1)$.
  Iterating the construction, after $r-1$ shears we arrive at the
  Abelian Lie algebra $\bR^m$, as claimed.

  Conversely, suppose we are given $\mfh = (\mfg/\mfa) \oplus \mfa$.
  Then $k\colon\mfh=(\mfg/\mfa) \oplus \mfa\rightarrow\mfg$,
  $k((X,Y)) \coloneqq p(X)+Y$ is a vector space isomorphism.
  A shear that recovers $\mfg$ is now given by the two-form
  $\tilde \omega \coloneqq k^*\omega \in \Lambda^2\mfh^* \otimes
  \mfa$, the one-form
  $\tilde \eta \coloneqq k^*\gamma \in \mfh^* \otimes \gl(\mfa)$,
  $\tilde{\xi}\coloneqq-\inc$ and $\tilde{a}\coloneqq\id_{\mfa}$, cf.\
  also Theorem~\ref{th:duality}
\end{proof}

\section{The shear}

\subsection{Lifting certain vector bundle morphisms}
\label{sec:lift}

Now we define the shear construction in full generality.  Motivated by
the last section, we start with a vector bundle $\pi_E\colon E\to M$
endowed with a flat connection $\nabla = \nabla^E$ and a vector bundle
morphism $\xi\colon E\to TM$ satisfying condition~\ref{item:bracket}
above, that is
\begin{equation}
  \label{eq:xinablatorsionfree}
  \xi(\nabla_{\xi e_1}e_2-\nabla_{\xi e_2}e_1) = [\xi e_1,\xi e_2].
\end{equation}
We will then say that $(\xi,\nabla)$ is \emph{torsion free}.
Moreover, we assume that we have a second vector bundle
$\pi_F\colon F\to M$ of the same rank with flat connection
$\nabla = \nabla^F$ and a two-form $\omega\in \Omega^2 (M,F)$ with
values in~$F$ such that $\dN \omega = 0$.  We do not require
condition~\ref{item:isotropic} here as it will naturally follow from
our set-up below.  Condition~\ref{item:invariance} will arise as the
appropriate invariance condition when we consider transferring
differential forms in~\S\ref{sec:differential-forms}.

Let us assume that $M$ is the leaf space of a foliation on some
manifold~$P$ with leaves of dimension
$\mathrm{rk}(E) = \mathrm{rk}(F)$.  Write $\pi\colon P\to M$ for the
projection, which is a surjective submersion.  We wish to identify the
pull-back of~$F$ to~$P$ with the tangent spaces to the leaves
of the foliation.

\begin{definition}
  \label{def:total-space}
  Suppose there are a vector bundle morphism $\rho\colon \pi^*F\to TP$
  and a one-form $\theta\in \Omega^1 (P,\pi^*F)$, so a bundle morphism
  $\theta\colon TP\to \pi^*F$, such that
  \begin{enumerate}[\upshape(1)]
  \item\label{item:theta-rho} $\theta \circ \rho = \id_{\pi^*F}$,
  \item\label{item:pi-rho} $d\pi \circ \rho = 0$ and
  \item\label{item:d-theta} $\dN \theta = \pi^*\omega$.
  \end{enumerate}
  Then we call $(P,\theta,\rho)$ a \emph{shear total space}
  for~$\omega$.  We will call $\dim P - \dim M$ the
  \emph{rank} of~$P$.
\end{definition}

It follows that the dimension of each leaf of the foliation on~$P$
is equal to the rank of~$F$.  We define the natural subbundles
\begin{equation*}
  \cH\coloneqq \ker \theta,\quad \cV\coloneqq \ker d\pi
\end{equation*}
of $TP$.  We note that our assumptions give $\cV = \rho(\pi^*F)$, so
$TP = \cH \oplus \cV$.  We call $\cH$ the \emph{horizontal} and $\cV$
the \emph{vertical} subbundles.

\begin{remark}\label{re:Ehresmannconnection}
  For a shear total space, conditions \ref{item:theta-rho} and
  \ref{item:pi-rho} identify the vertical subbundle $\cV$ with the
  flat vector bundle $\pi^*F$.
  In particular, the element $\hat\theta = \rho \circ \theta$ in
  $\Omega^1(P,\cV)=\mathrm{Hom}(TP,\cV)\subset \End(TP)$ is a
  projection onto the vertical subbundle $\cV$.
  Thus, if $P$ is actually a fibre bundle, then $\hat\theta$ is (the
  connection form of) an Ehresmann connection on~$P$.
  Recall that the \emph{curvature} $R\in \Omega^2(P,TP)$ of the
  Ehresmann connection $\hat\theta$ is $R(X,Y)=\hat\theta[X_H,Y_H]$
  for all $X,Y\in \mathfrak{X}(P)$, where $Z_H$ is the horizontal part
  of $Z\in \mathfrak{X}(P)$.
  Hence, condition \ref{item:d-theta} implies $R=-\rho\circ\pi^*\omega$
  since both forms are horizontal and
  $\dN \theta(X,Y)=\nabla_X (\theta\, Y)-\nabla_Y (\theta\,
  X)-\theta[X,Y]=-\theta[X,Y]$ for horizontal $X$ and $Y$.
\end{remark}

Suppose now that we are given an arbitrary Ehresmann connection
$\hat\theta$ on a fibre bundle $\pi\colon P\rightarrow M$ as in
Remark~\ref{re:Ehresmannconnection}.
Then the natural question arises when $\hat\theta$ gives rise to a shear
total space $(P,\theta,\inc)$.
This question is answered in the following

\begin{proposition}\label{pro:Ehresmannconnection}
  Let $\pi\colon P\rightarrow M$ be a fibre bundle such that the
  vertical subbundle $\cV$ is the pull-back of a flat vector bundle
  $(F,\nabla)$ over~$M$.
  Write $\inc\colon\pi^*F = \mathcal V \to TP$ for the inclusion map.
  Then an Ehresmann connection $\hat\theta\in \End(TP)$ gives rise to
  a shear total space $(P,\theta,\inc)$,
  $\hat\theta = \inc\circ\theta$, for some $\omega\in \Omega^2(M,F)$,
  if and only if
  \begin{equation}
    \label{eq:rhotorsionfree}
    [X_1,X_2] = \nabla_{X_1} X_2-\nabla_{X_2} X_1
  \end{equation}
  for all vertical $X_1,X_2\in \mathfrak{X}(P)$ and all local parallel
  vertical vector fields preserve the horizontal subbundle~$\cH$.
  If this is the case, then the curvature $R$ of $\hat\theta$
  satisfies $\dN \theta=\pi^*\omega=-R$.
\end{proposition}

\begin{proof}
  Note that there is an $\omega\in \Omega^2(M,F)$ with
  $\dN \theta = \pi^*\omega$ if and only if for all local parallel
  frames $(f^1,\dots,f^k)$ of~$F$ the forms
  $f^1\circ \dN\theta,\dots, f^k\circ \dN\theta$ are basic.
  As $d(f^i\circ \dN\theta) = f^i\circ \dN \dN\theta=0$ for all
  $i=1,\dots, k$, this is, in turn, equivalent to $\dN\theta$ being
  horizontal.

  To check the horizontality of $\dN\theta$, first let $X_1$, $X_2$ be
  two vertical vector fields on~$P$.  Then
  \begin{equation*}
    \begin{split}
      \dN \theta(X_1,X_2)
      &= \nabla_{X_1}(\theta\,X_2) - \nabla_{X_2}(\theta\,X_1)
        -\theta([X_1,X_2])\\
      &= \nabla_{X_1} X_2-\nabla_{X_2} X_1-[X_1,X_2]
    \end{split}
  \end{equation*}
  as $[X_1,X_2]$ is vertical.
  So $\dN\theta(X_1,X_2)=0$ if and only if equation
  \eqref{eq:rhotorsionfree} holds.
  Next, let $X$, $Y$ be vector fields on~$P$ with $X$ vertical and $Y$
  horizontal.
  As $\cV=\pi^*F$ has a local basis of parallel sections, we may
  assume that $X$ is parallel.
  Hence, $\dN \theta(X,Y)=-\theta([X,Y])$, which is zero if and only
  if $\cL_X Y=[X,Y]$ is horizontal, i.e.\ if and only if $X$ preserves
  the horizontal subbundle.
\end{proof}

\begin{remark}\label{re:fibresaffinemfds}
  By equation~\eqref{eq:rhotorsionfree}, the fibres of
  $\pi\colon P\rightarrow M$ are endowed with a torsion-free flat
  connection, so they are affine manifolds.
  Moreover, $\cV$ is the pull-back of a vector bundle over $M$, so the
  fibres are parallelisable.
  As the connection is a pull-back, there is a parallel, and so
  commuting, basis of vector fields on each fibre.
  This is in accordance with the twist case where the fibres were
  connected Abelian Lie groups.

  The proof of Proposition~\ref{pro:Ehresmannconnection} shows that
  for an arbitrary shear total space $(P,\theta,\rho)$ we have
  $[\rho f_1,\rho f_2]=\rho(\nabla_{\rho f_1}f_2-\nabla_{\rho
  f_2}f_1)$ for all $f_1,f_2\in\Gamma(\pi^*F)$, which is the
  torsion-free condition~\eqref{eq:xinablatorsionfree}.
  Finally, note that the condition that (local) parallel vertical
  vector fields preserve the horizontal subbundle in
  Proposition~\ref{pro:Ehresmannconnection} corresponds in the twist
  case to the principal action preserving the horizontal subbundle.
\end{remark}

Now we want to find conditions under which there exists a vector
bundle morphism $\mathring{\xi}\colon \pi^*E\to TP$
covering~$\xi$, i.e.
\begin{equation}\label{eq:cover}
  \begin{CD}
    \pi^* E @>\mathring{\xi}>> TP \\
    @VVV @VV d\pi V\\
    E @>\xi>> TM.
  \end{CD}
\end{equation}
commutes, such that $\mathring\xi$ preserves $\theta$:
\begin{equation}
  \label{eq:invariantconnection}
  \cL_{\mathring{\xi}}^{\nabla} \theta
  \coloneqq \dN(\mathring{\xi}\hook \theta) +
  \mathring{\xi}\hook \dN\theta
  = 0
\end{equation}
and $(\mathring\xi,\nabla)$ is torsion-free:
\begin{equation}\label{eq:mathringxi}
  \mathring{\xi}(\nabla_{\mathring{\xi} \tilde{e}_1}\tilde{e}_2
  - \nabla_{\mathring{\xi}\tilde{e}_2}\tilde{e}_1)
  = [\mathring{\xi}\tilde{e}_1,\mathring{\xi}\tilde{e}_2]
\end{equation}
for all $\tilde{e}_1, \tilde{e}_2\in \Gamma(\pi^*E)$.

Let us motivate our interest in such a bundle map~$\mathring{\xi}$.
Firstly, the two conditions were true in the left-invariant case
discussed in the previous section and they hold for the map $\rho$ of
a shear total space, since
$\cL_{\rho}^{\nabla}\theta = \rho\hook \dN \theta+\dN \id_{\pi^*F} =
\rho \hook \pi^*\omega = 0$.
Furthermore, the torsion-free condition implies that
$\mathring{\xi}(\pi^*E)$ is involutive; we will see that it is
equivalent to involutivity under our assumptions.
When $\mathring{\xi}$ has constant rank, this is, in turn, equivalent to
the integrability of $\mathring{\xi}(\pi^*E)$.
The shear should then be the leaf space of the corresponding foliation
on~$P$.

For the first condition, note that
$\cL_{\mathring{\xi}}^{\nabla}\theta = 0$ may be written as
\begin{equation}\label{eq:Liethetazero}
  \nabla_{\mathring{\xi}\tilde{e}}(\theta X)
  = \theta\bigl(\mathring{\xi}(\nabla_X \tilde{e})
  + [\mathring{\xi}\tilde{e},X]\bigr)
\end{equation}
for all vector fields $X\in \vf(P)$ and sections
$\tilde{e}\in \Gamma(\pi^*E)$.  As $E$ and $F$ are flat, they
have local bases of parallel sections.  Take $\tilde{e} = \pi^*e$ for
some local parallel section~$e$ of $E$.  If $X$ is horizontal,
this gives $\theta([\mathring{\xi}(\pi^*e),X]) = 0$, so
$\mathring{\xi}(\pi^*e)$ preserves the horizontal space, just as the
lifted action does in the twist construction.  If $X$ is a local
vertical vector field with $\theta X$ is parallel, we obtain that
$[\mathring{\xi}(\pi^*e),X]$ is horizontal.  However,
$\mathring{\xi}(\pi^*e)$ is $\pi$-related to $\xi e$ and $X$ is
$\pi$-related to~$0$, so the commutator has to be vertical too.
This shows
\begin{equation}
  \label{eq:commuting-actions}
  [\mathring{\xi}(\pi^*e),X] = 0,
\end{equation}
which in the twist construction is the requirement that the lifted and
principal actions commute.

After this motivation, we are interested in expressing the
requirements for a lift $\mathring{\xi}$ as above in equivalent
conditions for data on~$M$.

\begin{theorem}
  \label{th:liftingbundlemorphisms}
  Under the above assumptions \eqref{eq:xinablatorsionfree} and
  Definition~\ref{def:total-space}\ref{item:theta-rho}--\ref{item:d-theta},
  there exists a vector bundle morphism
  $\mathring{\xi}\colon \pi^*E\to TP$ covering $\xi\colon E\to TM$,
  preserving~$\theta$ and with $(\mathring\xi,\nabla)$
  torsion-free if and only if $\cL_{\xi}^{\nabla} \omega = 0$,
  $\xi\hook \omega$ is $\dN$-exact and $\xi^*\omega = 0$.
\end{theorem}

\begin{proof}
  Using $\cH$, we can lift $\xi\colon E\to TP$ uniquely to a bundle
  morphism $\tilde{\xi}\colon \pi^*E\to TP$ covering~$\xi$ with
  $\tilde{\xi}(\pi^*E)\subset \cH$.  When $\mathring\xi$ exists we
  get $d\pi(\mathring{\xi}-\tilde{\xi})=0$, so
  $(\mathring{\xi}-\tilde{\xi})(\pi^*E)\subset \cV$.  As
  $\rho\colon \pi^*F\to TP$ is injective and $\rho(\pi^*F) = \cV$,
  there is a uniquely defined bundle map
  $\mathring{a}\colon \pi^*E\to \pi^*F$ with
  \begin{equation}
    \label{eq:xi-rho-a}
    \mathring{\xi} = \tilde{\xi} + \rho\circ \mathring{a}.
  \end{equation}
  Thus,
  \begin{equation*}
    \cL_{\mathring{\xi}}^{\nabla} \theta
    = \dN(\mathring{\xi}\hook \theta)
    + \mathring{\xi}\hook \dN\theta
    = \dN\mathring{a}+\mathring{\xi}\hook \pi^*\omega
    = \dN\mathring{a} + \pi^*(\xi\hook \omega),
  \end{equation*}
  and so $\cL_{\mathring{\xi}}^{\nabla}\theta = 0$ if and only if
  $\pi^*(\xi\hook \omega) = -\dN\mathring{a}$.  But then
  $\nabla_{\rho} \mathring{a} = \rho\hook \dN\mathring{a} =
  -\rho\hook\pi^*(\xi\hook \omega) = 0$.  Thus
  $\nabla_{\cV} \mathring{a} = 0$.  This implies that
  $\mathring a$ is basic, so $\mathring{a} = \pi^* a$ for some
  bundle map $a\colon E\to F$.  Thus,
  $\cL_{\mathring{\xi}}^{\nabla}\theta = 0$ if and only if there
  exists a bundle map $a\colon E\to F$ with
  \begin{equation}
    \label{eq:xi-d-a}
    \xi\hook \omega=-\dN a,
  \end{equation}
  which says $\xi\hook \omega$ is $\dN$-exact.  Conversely,
  given~\eqref{eq:xi-d-a}, we immediately get
  $\cL^{\nabla}_{\xi}\omega = \dN(\xi\hook \omega)+\xi\hook \dN \omega
  = 0$, since $(\dN)^2 = 0$ and $\dN\omega = 0$ and may construct
  $\mathring\xi$ via~\eqref{eq:xi-rho-a} with
  $\cL_{\mathring{\xi}}^{\nabla}\theta=0$.

  Now we compute
  $[\mathring{\xi}\tilde{e}_1,\mathring{\xi}\tilde{e}_2]$ for two
  sections $\tilde{e}_1,\tilde{e}_2\in \Gamma(\pi^*E)$.  It suffices
  to consider $\tilde{e}_i = \pi^*e_i$, $i = 1,2$, for
  $e_1, e_2\in \Gamma(E)$.  Denote by $X^\sim\in \vf(p)$ the
  horizontal lift of a vector field $X\in \vf(M)$ and observe that
  $\tilde{\xi}(\pi^*e_i) = (\xi e_i)^\sim$ for $i = 1,2$.  Hence,
  \begin{equation*}
    \begin{split}
      [\mathring{\xi}(\pi^*e_1),\mathring{\xi}(\pi^*e_2)]
      &= [(\xi e_1)^\sim,(\xi e_2)^\sim]
        + [\rho\pi^*(ae_1),(\xi e_2)^\sim] \eqbreak
        + [(\xi e_1)^\sim,\rho\pi^*(ae_2)]
        + [\rho\pi^*(ae_1),\rho\pi^*(ae_2)].
    \end{split}
  \end{equation*}
  Thus, using \eqref{eq:xinablatorsionfree}, the horizontal part of
  $[\mathring{\xi}(\pi^*e_1),\mathring{\xi}(\pi^*e_2)]$ is equal to
  \begin{equation*}
    [\xi(e_1),\xi(e_2)]^\sim
    = \bigl(\xi(\nabla_{\xi e_1}e_2-\nabla_{\xi e_2}e_1)\bigr)^\sim
    = \tilde{\xi}\bigl(\pi^*(\nabla_{\xi e_1}e_2-\nabla_{\xi e_2}e_1)\bigr).
  \end{equation*}
  To compute the vertical part, we consider
  \begin{equation*}
    \begin{split}
      \theta([\tilde{\xi}\pi^*e_1,\tilde{\xi}\pi^*e_2])
      &= -\dN\theta(\tilde{\xi}\pi^*e_1,\tilde{\xi}\pi^*e_2)
        + \nabla_{\tilde{\xi}\pi^*e_1}(\theta(\tilde{\xi}\pi^*e_2))
        - \nabla_{\tilde{\xi}\pi^*e_2}(\theta(\tilde{\xi}\pi^*e_1))\\
      &= -(\pi^*\omega)(\tilde{\xi}\pi^*e_1,\tilde{\xi}\pi^*e_2)
        = -\pi^*(\omega(\xi e_1,\xi e_2))
    \end{split}
  \end{equation*}
  and
  \begin{equation*}
    \begin{split}
      \theta([\rho\pi^*(ae_1),\tilde{\xi}\pi^*e_2])
      &= -\dN\theta(\rho\pi^*(ae_1),\tilde{\xi}\pi^*e_2)
        - \nabla_{\tilde{\xi}\pi^*e_2}(\theta\rho\pi^*(ae_1))\\
      &= -(\pi^*\omega)(\rho\pi^*(ae_1),\tilde{\xi}\pi^*e_2)
        - \nabla_{\tilde{\xi}\pi^*e_2}(\pi^*(ae_1))\\
      &= -\pi^*(\nabla_{\xi e_2} (a e_1))
        = -\pi^*((\nabla_{\xi e_2} a)e_1 + a(\nabla_{\xi e_2} e_1))\\
      &= -\pi^*((\dN a) (\xi e_2)e_1 + a(\nabla_{\xi e_2} e_1))\\
      &= \pi^*(\omega(\xi e_1,\xi e_2) - a(\nabla_{\xi e_2}e_1)).
    \end{split}
  \end{equation*}
  Similarly,
  $\theta([\tilde{\xi}\pi^*e_1, \rho\pi^*(ae_2)]) = \pi^*(\omega(\xi
  e_1,\xi e_2) + a(\nabla_{\xi e_1} e_2))$.  Furthermore,
  \begin{equation*}
    \begin{split}
      \theta([\rho\pi^*(ae_1),\rho\pi^*(ae_2)])
      &= \nabla_{\rho\pi^*(ae_1)} (\pi^*(ae_2))
        - \nabla_{\rho\pi^*(ae_2)} (\pi^*(ae_1))\\
      &= \pi^*(\nabla_{d\pi(\rho\pi^*(ae_1))} (ae_2)
        - \nabla_{d\pi(\rho\pi^*(ae_2))} (ae_1))\\
      &= 0
    \end{split}
  \end{equation*}
  as $d\pi\circ \rho = 0$.  Now $\rho\circ \theta$ is the projection
  onto the vertical part and so
  \begin{equation*}
    \begin{split}
      [\mathring{\xi}\pi^*e_1,\mathring{\xi}\pi^*e_2]
      &= \tilde{\xi}\pi^*(\nabla_{\xi e_1 }e_2-\nabla_{\xi e_2}e_1)
        + \rho\pi^*(\omega(\xi e_1,\xi e_2)) \eqbreak
        +\rho\pi^*(a(\nabla_{\xi e_1} e_2-\nabla_{\xi e_2} e_1))\\
      &= \tilde{\xi}\bigl(\nabla_{\mathring{\xi}\pi^*e_1} (\pi^*e_2)
        - \nabla_{\mathring{\xi}\pi^* e_2}  (\pi^* e_1)\bigr)
        + \rho\pi^*(\xi^*\omega)(e_1,e_2) \eqbreak
        + (\rho\circ \mathring{a})\bigl(\nabla_{\mathring{\xi} \pi^* e_1}
        (\pi^* e_2) - \nabla_{\mathring{\xi} \pi^* e_2}  (\pi^* e_1)\bigr)\\
      &=
        \mathring{\xi}\bigl(\nabla_{\mathring{\xi}\pi^*e_1}(\pi^*e_2) -
        \nabla_{\mathring{\xi}\pi^*e_2}(\pi^*e_1)\bigr)
        + \rho\pi^*(\xi^*\omega)(e_1,e_2).
    \end{split}
  \end{equation*}
  As $\rho$ and $\pi^*$ are injective, we have
  $\mathring{\xi}(\nabla_{\mathring{\xi}\tilde{e}_1}\tilde{e}_2 -
  \nabla_{\mathring{\xi}\tilde{e}_2}\tilde{e}_1) =
  [\mathring{\xi}\tilde{e}_1,\mathring{\xi}\tilde{e}_2]$ for all
  $\tilde{e}_1, \tilde{e}_2\in \Gamma(\pi^*E)$ if and only if
  $\xi^*\omega = 0$.
\end{proof}

As noted in the proof, if $\xi\hook \omega$ is $\dN$-exact and
$\omega$ is $\dN$-closed, then we automatically get
$\cL_{\xi}^{\nabla}\omega = 0$.  Hence,
Theorem~\ref{th:liftingbundlemorphisms} naturally leads to the
following definition.

\begin{definition}\label{def:shear}
  \emph{Shear data} on a smooth manifold~$M$ is a triple
  $(\xi,a,\omega)$ consisting of a bundle map $\xi\colon E\to TM$, an
  invertible bundle morphism $a\colon E\to F$ and a two-form
  $\omega\in \Omega^2(M,F)$ with values in $F$, where $E$ and $F$ are
  flat vector bundles over $M$ of the same rank and
  \begin{enumerate}[\upshape(i)]
  \item\label{item:s-bracket} $(\xi,\nabla)$ is torsion-free
    \eqref{eq:xinablatorsionfree},
  \item\label{item:s-closure} $\dN\omega = 0$,
  \item\label{item:s-contract} $\xi\hook \omega=-\dN a$ and
  \item\label{item:s-pullback} $\xi^*\omega = 0$.
  \end{enumerate}

  Suppose additionally there is a shear total space
  $(P,\theta,\rho)$ as in Definition~\ref{def:total-space}.
  Define $\mathring{\xi}\colon \pi^*E\to TP$ by
  $\mathring{\xi}\coloneqq \tilde{\xi}+\rho\circ \pi^*a$, with
  $\tilde{\xi}$ the horizontal lift of $\xi$.  Then, by
  Theorem~\ref{th:liftingbundlemorphisms}, $\mathring{\xi}(\pi^*E)$ is
  an integrable distribution.  Furthermore, invertibility of $a$
  ensures that it is of constant rank.  The leaf space
  \begin{equation*}
    S\coloneqq P/\mathring{\xi}(\pi^*E)
  \end{equation*}
  is called \emph{the shear of $(M,\xi,a,\omega)$} when it is a smooth
  manifold.
\end{definition}

\begin{remark}\label{re:shear-is-twistlocally}
  Locally the shear is essentially the twist construction.  Given
  shear data $(\xi,a,\omega)$ on a manifold $M$, choose parallel
  frames $(e_1,\dots,e_k)$ and $(f_1,\dots,f_k)$ of $E$ and $F$
  locally.  Then we may write $\omega = \sum_{i = 1}^k \omega_i f_i$
  and $a = \sum_{i,j = 1}^k a_i^j e^i\otimes f_j$.  Putting
  $\mfa_M = \mfa_P\coloneqq \bR^k$,
  $\Omega\coloneqq (\omega_1,\dots,\omega_k)\in \Omega^2(M,\mfa_P)$,
  $A\coloneqq (a_i^j)_{i,j = 1,\dots,k}\in C^{\infty}(M,\mfa_P\otimes
  \mfa_M^*)$ and $\Xi\colon \mfa_M\to \vf(M)$,
  $\Xi(x^1,\dots,x^m)\coloneqq \sum_{i = 1}^k x^i \xi(e_i)$,
  conditions \ref{item:s-bracket}--\ref{item:s-pullback} of
  Definition~\ref{def:shear} show that $(\Xi,\Omega,A)$ is local twist
  data.  Note that these identifications depend on the choices of the
  flat structures.
\end{remark}

\begin{remark}
  In some cases we may construct a suitable space $P$ via a
  principal bundle on the universal cover~$\tilde M$ of~$M$.
  The pull-back of~$F$ to $\tilde M$ is flat and has a global
  basis of parallel sections.  The pull-back of $\omega$ may
  then be interpreted as the curvature two-form for a principal bundle
  $\tilde P \to \tilde M$ with Abelian structure group.  This
  imposes integrality conditions on the pull-back of~$\omega$, and
  one then needs to investigate whether $\tilde P$ can be chosen
  so that it descends to a bundle over~$M$.  We examined these
  questions in detail for one-dimensional fibres in~\cite{FS}.
  Consideration of similar questions phrased in the language of Lie
  algebroids may be found in \cite{McK1,McK2}.  However, as will see
  later, candidate spaces~$P$ may arise in other ways, unrelated
  to principal bundles.
\end{remark}

\subsection{Differential forms}
\label{sec:differential-forms}

Let us now fix some shear data $(\xi,a,\omega)$ on a manifold $M$ and
a corresponding triple $(\pi\colon P\to M,\theta,\rho)$ as in
Definition~\ref{def:shear} such that the shear
$S = P/\mathring{\xi}(\pi^*E)$ is smooth.
Write $\pi_S\colon P \to S$ for the projection to $S$.
We are interested in how to move geometric structures on~$M$ to~$S$.
In particular, how to relate $(p,0)$-tensor fields on~$M$ with
$(p,0)$-tensor fields on~$S$.

\begin{definition}
  Let $\alpha$ be a $(p,0)$-tensor field on $M$ and $\alpha_S$ be a
  $(p,0)$-tensor field on $S$.  We say that $\alpha$ is
  \emph{$\cH$-related to} $\alpha_S$, in symbols
  \begin{equation*}
    \alpha \Hrel \alpha_S,
  \end{equation*}
  if
  \begin{equation*}
    \pi^*\alpha|_{\cH} = \pi_S^*\alpha_S|_{\cH}.
  \end{equation*}
\end{definition}

For differential forms, this relation may be concretely described.

\begin{proposition}\label{pro:uniqueHrelated}
  A $k$-form $\alpha$ on $M$ is $\cH$-related to some $k$-form
  $\alpha_S$ on the shear $S$ if and only if
  $\cL_{\xi}^{\nabla} \alpha = 0$.  In this case, the $k$-form
  $\alpha_S$ is uniquely determined by
  \begin{equation*}
    \pi_S^* \alpha_S = \pi^* \alpha+\sum_{i = 1}^k
    (-1)^{i(2k+1-i)/2} \pi^*\bigl((\xi\circ
    a^{-1}\hook)^i \alpha\bigr) \wedge \theta^i.
  \end{equation*}
\end{proposition}

\begin{proof}
  Suppose that $\alpha\in \Omega^k M$ is $\cH$-related to
  $\alpha_S\in \Omega^k S$.  If we decompose $\pi_S^*\alpha_S$ with
  respect to $\cH$ and $\cV$, we get by definition
  \begin{equation*}
    \pi_S^*\alpha_S = \pi^*\alpha+\sum_{i = 1}^k \beta_i\wedge \theta^i
  \end{equation*}
  for certain
  $\beta_i\in \Omega^{k-i}\left(P,\Lambda^i (\pi^*F)^*\right)$ with
  $\beta_i|_\cV = 0$.  Here $\theta^i$ denotes the element
  $\Lambda^i\theta\in\Gamma(\End(\Lambda^i TP,\Lambda^i
  \pi^*F))\cong\Omega^i (P,\Lambda^i \pi^* F)$ pointwise induced by
  $\theta\in \Omega^1 (P,\pi^*F)\cong \Gamma(\End(TP,\pi^*F))$.  For
  $X_1,\dots,X_i\in \vf(P)$, this means
  $\theta^i(X_1\wedge\dots\wedge X_i) = \theta(X_1)\wedge \dots \wedge
  \theta(X_i)$.

  As $\mathring{\xi}\hook \pi_S^*\alpha_S = 0$, we get
  \begin{equation*}
    0 = \pi^* (\xi\hook \alpha)
    + \sum_{i = 1}^k (\tilde{\xi}\hook \beta_i)\wedge \theta^i
    + \sum_{i = 1}^k (-1)^{k-i} \beta_i \wedge ((\rho\circ \pi^*
    a)\hook \theta^i)
  \end{equation*}
  As $(\rho\circ\pi^*a)\hook\theta^i = \pi^*a\,\theta^{i-1}$, this
  tells us $\pi^*(\xi\hook \alpha) = (-1)^k \beta_1\pi^*a$ and
  $(\tilde{\xi}\hook \beta_{i-1})\wedge \theta^{i-1} = (-1)^{k-i+1}
  \beta_i\wedge \pi^*a\, \theta^{i-1}$ for $i = 2,\dots,k$.  It
  follows that
  \begin{equation*}
    \beta_1 = (-1)^k \pi^*\left(\left(\xi\circ a^{-1}\right)\hook
      \alpha\right)
  \end{equation*}
  and
  \begin{equation*}
    \begin{split}
      \MoveEqLeft \beta_i(X_1,\dots,X_{k-i})(\pi^*a(e),f_1,\dots,f_{i-1})\\
   & = (-1)^{k-i+1} \beta_{i-1}
     \bigl(\tilde{\xi}(e),X_1,\dots,X_{k-i}\bigr)(f_1,\dots,f_{i-1})
    \end{split}
  \end{equation*}
  for $e\in \Gamma(\pi^*E)$, $X_1,\dots,X_{k-i}\in \vf(P)$,
  $f_1,\dots,f_{i-1}\in \Gamma(\pi^*F)$ and $i = 2,\dots,k$.  As
  $a$ is invertible, we may write $e = \pi^*a^{-1}(f_0)$ and
  see that the last condition says
  \begin{equation*}
    \beta_i = (-1)^{k-i+1}(\tilde{\xi}\circ \pi^*a^{-1}) \hook \beta_{i-1},
  \end{equation*}
  for $i = 2,\dots,k$.  Note that there is also the implicit statement
  that both sides are anti-symmetric when evaluated on $i$
  sections of $\pi^*F$.  Recursion now gives
  \begin{equation*}
    \beta_i = (-1)^{i(2k+1-i)/2} \pi^*\bigl((\xi\circ a^{-1})^i\hook \alpha\bigr)
  \end{equation*}
  and this is, in fact, anti-symmetric in sections of $\pi^*F$.  Thus the
  claimed formula for $\alpha_S$ holds, when $\alpha_S$ exists.

  So, to finish the proof, we have to consider the
  semi-basic $k$-form
  \begin{equation*}
    \bWalpha\coloneqq \pi^* \alpha
    + \sum_{i = 1}^k (-1)^{i(2k+1-i)/2} \pi^*\bigl((\xi\circ
    a^{-1}\hook)^i \alpha\bigr) \wedge \theta^i
  \end{equation*}
  and have to determine when this is basic for~$\pi_S$ as exactly then
  $\bWalpha$ will be the pull-back of a unique $\alpha_S$
  on~$S$.  For $\bWalpha$ to be basic requires
  $\cL_{\mathring\xi e}\bWalpha = 0$ for each
  $e \in \Gamma(\pi^*E)$.  This condition is
  \begin{equation*}
    0 = \cL_{\mathring{\xi} e} \bWalpha
    = \mathring{\xi}e\hook d\bWalpha
    = \mathring{\xi}e\hook d\bWalpha + \dN(\mathring{\xi}\hook
    \bWalpha)e
    = (\cL_{\mathring{\xi}}^{\nabla} \bWalpha)e,
  \end{equation*}
  which just says $\cL^\nabla_{\mathring\xi}\bWalpha = 0$.  Now
  $\cL_{\mathring{\xi}}^{\nabla}\theta = 0$, so we have
  \begin{equation*}
    \cL_{\mathring{\xi}}^{\nabla} \bWalpha
    = \pi^*(\cL_{\xi}^{\nabla}\alpha)
    + \sum_{i = 1}^k (-1)^{i(2k+1-i)/2}
    \pi^*\Bigl(\cL_{\xi}^{\nabla}\bigl((\xi\circ a^{-1}\hook)^i
    \alpha\bigr)\Bigr)\wedge \theta^i.
  \end{equation*}
  Taking the horizontal part, we see that
  $\cL_{\mathring{\xi}}^\nabla \bWalpha = 0$ implies
  $\cL_{\xi}^{\nabla}\alpha = 0$.

  Conversely, assume that $\cL_{\xi}^{\nabla}\alpha = 0$.  For an
  $\ell$-form $\tau$ with values in $(F^*)^{\otimes r}$ and local
  parallel sections $e$ of~$E$, $f_j$ of~$F$, note
  $(\cL_{\xi}^{\nabla}\tau) (e,f_1,\dots,f_r) = \cL_{\xi e}
  (\tau(f_1,\dots,f_r))$.  Writing $e_i\coloneqq a^{-1}(f_i)$,
  $\xi(i) = \xi(e_i)$ and $\iota_X = X\hook$, we find
  \begin{equation*}
    \begin{split}
      \MoveEqLeft
      \bigl(\cL_{\xi}^{\nabla}\bigl((\xi\circ a^{-1}\hook)^i
      \alpha\bigr)\bigr) (e,f_1,\dots,f_i) \\
   &=  \cL_{\xi e} (\iota_{\xi(i)}\dots \iota_{\xi(1)}\alpha)\\
   &= \iota_{\xi(i)} \cL_{\xi e} (\iota_{\xi({i-1})} \dots
     \iota_{\xi(1)}\alpha)
     + \iota_{[\xi e,\xi(i)]}\iota_{\xi({i-1})}\dots
     \iota_{\xi(1)}\alpha\\
   &= \iota_{\xi(i)}\dots \iota_{\xi(1)}\cL_{\xi e}\alpha
     + \sum_{j = 1}^i (-1)^{j-1} \iota_{\xi(i)} \dots
     \widehat{\iota_{\xi(j)}} \dots \iota_{\xi(1)}
     \iota_{[\xi e,\xi(j)]} \alpha\\
   &= \iota_{\xi(i)}\dots \iota_{\xi(1)}(\cL_{\xi}^{\nabla}\alpha)e
     +\sum_{j = 1}^i (-1)^{j-1} \iota_{\xi(i)} \dots
     \widehat{\iota_{\xi(j)}} \dots \iota_{\xi(1)}
     \iota_{\xi(\nabla_{\xi e} e_j-\nabla_{\xi(j)} e)} \alpha\\
   &= \sum_{j = 1}^i (-1)^{j-1} \iota_{\xi(i)} \dots
     \widehat{\iota_{\xi(j)}} \dots \iota_{\xi(1)}
     \iota_{\xi(\nabla_{\xi e} (a^{-1} f_j))}\alpha
    \end{split}
  \end{equation*}
  using the torsion-free condition \eqref{eq:xinablatorsionfree} in
  the penultimate step.  But we compute
  \begin{equation*}
    \begin{split}
      \nabla_{\xi e}(a^{-1}f_j) &= (\nabla_{\xi e} a^{-1})f_j
                                  = - (a^{-1}\circ \nabla_{\xi e} a\circ a^{-1})f_j\\
                                &= a^{-1}\bigl((\xi\hook \omega)(\xi
                                  e)(e_j)\bigr) = a^{-1}(\omega(\xi e_j,\xi e)) = 0,
    \end{split}
  \end{equation*}
  and so
  $\bigl(\cL_{\xi}^{\nabla}\bigl((\xi\circ a^{-1}\hook)^i
  \alpha\bigr)\bigr) (e,f_1,\dots,f_i) = 0$.  As the local parallel
  sections span $E$ and $F$, we get
  $\cL_{\xi}^{\nabla}\bigl((\xi\circ a^{-1}\hook)^i \alpha\bigr) = 0$
  and conclude that $\cL_{\xi}^{\nabla}\bWalpha = 0$.
\end{proof}

Next we consider the differentials of $\cH$-related $k$-forms:

\begin{corollary}\label{cor:differentials}
  Let $\alpha\in \Omega^k M$ be a $k$-form on $M$ with
  $\cL_{\xi}^{\nabla} \alpha = 0$ and take $\alpha_S\in \Omega^k S$
  with $\alpha \Hrel \alpha_S$.  Then
  \begin{equation*}
    d\alpha-\bigl(\xi\circ a^{-1}\hook \alpha\bigr)\wedge \omega\Hrel d\alpha_S.
  \end{equation*}
\end{corollary}

\begin{proof}
  As $\cL_{\xi}^{\nabla} \alpha = 0$, $\cL_{\xi}^{\nabla}\omega = 0$
  and $\cL_{\xi}^{\nabla}(\xi\circ a^{-1}\hook \alpha)=0$ by the proof
  of Proposition~\ref{pro:uniqueHrelated}, we have
  $\cL_{\xi}^{\nabla}\bigl(d\alpha-(\xi\circ a^{-1}\hook \alpha)\wedge
  \omega\bigr)=0$.  So Proposition~\ref{pro:uniqueHrelated} tells us
  that there exists a unique $(k+1)$-form on $S$ which is
  $\cH$-related to
  $d\alpha-\left(\xi\circ a^{-1}\hook \alpha\right)\wedge \omega$.
  This $(k+1)$-form has to be equal to $d \alpha_S$ as differentiating
  the formula for $\pi_S^*\alpha_S$ in
  Proposition~\ref{pro:uniqueHrelated} gives us horizontally
  \begin{equation*}
    \pi_S^* d \alpha_S
    = \pi^* d\alpha - \pi^*\left(\xi\circ a^{-1}\hook \alpha\right)
    \wedge \dN\theta
    = \pi^*\left(d\alpha-\left(\xi\circ a^{-1}\hook \alpha\right)
      \wedge \omega\right),
  \end{equation*}
  as claimed.
\end{proof}

\subsection{Vector fields and almost complex structures}
\label{subsec:vfsacs}

With the notation from the previous section, we say that a vector
field $X$ on~$M$ is \emph{$\cH$-related} to a vector field $X_S$
on~$S$ if and only if their horizontal lifts $\widetilde{X}$,
$\widehat{X_S}$ to~$P$ agree.

For given $X\in \vf(M)$ there is at most one $X_S\in \vf(S)$ with
$X\Hrel X_S$.  Furthermore, $X_S$ exists if and only if
$[\widetilde{X},\Gamma(\mathring{\xi}(\pi^*E))]\subset
\Gamma(\mathring{\xi}(\pi^*E))$.  Now from
equation~\eqref{eq:Liethetazero} we obtain that for given
$e\in \Gamma(E)$ the vertical component of
$[\widetilde{X},\mathring{\xi}(\pi^*e)]$ equals the vertical component
of
$\mathring{\xi}(\nabla_{\widetilde{X}} \pi^*e) =
\mathring{\xi}(\pi^*(\nabla_X e))$ and the horizontal component of
$[\widetilde{X},\mathring{\xi}(\pi^*e)]$ equals
$\widetilde{[X,\xi(e)]}$. So the existence of $X_S\in \vf(S)$ with
$X\Hrel X_S$ is equivalent to $[\xi(e),X]=-\xi(\nabla_X e)$ for all
$e\in \Gamma(E)$. Defining
$(\cL_{\xi}^\nabla X)(e) \coloneqq [\xi(e),X]+\xi(\nabla_X e)$ for
$e\in \Gamma(E)$, the $\cH$-related vector field $X_S$ exists
if and only if
\begin{equation*}
  \cL_{\xi}^{\nabla} X=0.
\end{equation*}
Note that $\cL_{\xi}^{\nabla}$ behaves as the usual Lie derivative
under contractions:
$(\cL_{\xi}^{\nabla} \alpha)(X) = \cL_{\xi}^{\nabla}
(\alpha(X))-\alpha(\cL_{\xi}^{\nabla}X)$ for all $X\in \vf(M)$ and all
$\alpha\in \Omega^1 M$.

After these preliminaries, we consider Lie brackets of $\cH$-related
vector fields:

\begin{lemma}\label{le:commutatorVF}
  Suppose the vector fields $X,Y\in \vf(M)$, $X_S,Y_S\in \vf(S)$
  satisfy $X\Hrel X_S$ and $Y\Hrel Y_S$.  Then the identity
  \begin{equation*}
    [X,Y]+\xi a^{-1}\omega(X,Y)\Hrel [X_S,Y_S]
  \end{equation*}
  holds.
\end{lemma}

\begin{proof}
  Considering the decomposition into horizontal and vertical subspaces
  of $\pi$, we have
  $[\widetilde{X},\widetilde{Y}] = \widetilde{[X,Y]}+\rho\theta
  [\widetilde{X},\widetilde{Y}]$.  For $\pi_S$, note that the
  projection to the vertical subspace of $\pi_S$ is given by
  $\mathring{\xi}\pi^*a^{-1}\theta = \tilde{\xi}
  \pi^*a^{-1}\theta+\rho\theta$.  This gives
  \begin{equation*}
    \begin{split}
      [\widetilde{X},\widetilde{Y}]
      &= [\widehat{X_S},\widehat{Y_S}]
        = \widehat{[X_S,Y_S]}
        + \tilde{\xi}\pi^*a^{-1}\theta[\widehat{X_S},\widehat{Y_S}]
        + \rho\theta [\widehat{X_S},\widehat{Y_S}]\\
      &= \widehat{[X_S,Y_S]}
        - \tilde{\xi}\pi^*(a^{-1}\omega(X,Y))
        + [\widetilde{X},\widetilde{Y}]
        - \widetilde{[X,Y]}.
    \end{split}
  \end{equation*}
  Thus
  $\bigl([X,Y]+\xi a^{-1}\omega(X,Y)\bigr)^{\sim} =
  \widehat{[X_S,Y_S]}$ as claimed.
\end{proof}

Next, we consider almost complex structures $I$ on $M$ and $I_S$ on
$S$ and say that they are $\cH$-related if and only if for all
$p\in P$ and all $v\in T_{\pi(p)}M$ we have
$\widetilde{Iv}=(I_S d\pi_S\tilde{v})\spwhat$.  Then, given an almost
complex structure $I$ on $M$, there exists an $\cH$-related almost
complex structure on $S$ if and only if $\cL_{\xi}^{\nabla} I=0$,
where we extend $\cL_{\xi}^{\nabla}$ in the usual way to
$(1,1)$-tensors.  Moreover, the definition of the Nijenhuis tensor
$N_I$~of an almost complex structure $I$ and
Lemma~\ref{le:commutatorVF} yield:

\begin{proposition}\label{pro:Nijenhuis}
  For almost complex structures $I$ on $M$ and $I_S$ on $S$ with
  $I\Hrel I_S$, the Nijenhuis tensors are related by
  \begin{equation*}
    N_I - \mathcal{F} - I\mathcal{F}(I\any,\any)
    - I\mathcal{F}(\any,I\any) + \mathcal{F}(I\any,I\any)
    \Hrel N_{I_S},
  \end{equation*}
  where $\mathcal{F} = \xi a^{-1} \omega$.  \qed
\end{proposition}

\subsection{Duality}
\label{subsec:duality}
In this section, we assume that we are in the situation of
\S\ref{sec:lift} and show how we can then invert the shear
construction.

A first necessary condition to invert the shear is to construct flat
vector bundles $E_S$ and $F_S$ over the shear $S$ such that
$\pi_S^*E_S \cong \pi^*F$ and $\pi_S^*F_S \cong \pi^*E$ as bundles with flat
connections. If we have such
vector bundles and if we fix identifications $\pi_S^*E_S \cong \pi^*F$ and $\pi_S^*F_S \cong \pi^*E$,
we may extend $\Hrel$ to $k$-forms and vector fields
with values in tensor powers of these bundles. For example, we say that
$\alpha\in \Omega^k(M,E^{\otimes r}\otimes F^{\otimes s})$ is
\emph{$\cH$-related} to
$\alpha_S\in \Omega^k(S,E_S^{\otimes s}\otimes F_S^{\otimes r})$ if
and only if $\pi^*\alpha|_{\cH}=\pi_S^*\alpha_S|_{\cH}$.

\begin{theorem}\label{th:duality}
  Suppose $(M,E,\xi,F,a,\omega)$ shears via a total space
  $(P,\theta,\rho)$ to~$S$.  Suppose in addition that
  $\pi^*E$ and $\pi^*F$ can be trivialised by flat sections in
  a neighbourhood of each leaf of~$\mathring\xi$ on~$P$.  Then
  there is shear data $(\xi_S, a_S,\omega_S)$ on~$S$ realising
  $M$ as the shear of~$S$.

  More precisely, there are flat bundles $E_S,F_S \to S$, a
  torsion-free bundle map $\xi_S\colon E_S \to TS$, a two-form
  $\omega_S \in \Omega^2(S,F_S)$ and an
  $a_S \in \Omega^0(S,E_S^*\otimes F_S)$ such that
  \begin{enumerate}[\upshape(a)]
  \item as flat bundles $\pi_S^*E_S \cong \pi^*F$ and
    $\pi_S^*F_S \cong \pi^*E$,
  \item $a^{-1} \Hrel a_S$, $a^{-1}\omega \Hrel \omega_S$,
    $- \xi \circ a^{-1} \Hrel \xi_S$,
  \item $M$ is the shear of $S$
    via~$(P, \theta_S = (\pi^*a^{-1})\theta, \rho_S = \mathring\xi)
   $.
  \end{enumerate}
\end{theorem}

\begin{proof}
  Recall that the fibres of $\pi_S\colon P \to S$ are the leaves
  of $\mathring\xi$ and that by assumption $\pi^*F$ can be
  trivialised by flat sections in a neighbourhood of each leaf. So we
  may define a locally free sheaf $\mathcal{F}$ on~$S$ by letting
  $\mathcal{F}(U)$ be those sections $\sigma$ of $\pi^*F$ over
  $\pi_S^{-1}(U)$ which are constant on the leaves under a trivialisation
        by flat sections.

  Let $E_S$ be the associated vector bundle; by construction we have
  $\pi_S^*E_S = \pi^*F$.  We give $E_S$ a flat
  connection~$\nabla$ by declaring a local section to be parallel
  if its pull-back to $\pi^*F$ is parallel.  In a similar way, we
  construct a flat bundle $F_S \to S$ with
  $\pi_S^*F_S = \pi^*E$.

  Computing
  \begin{equation*}
    \cL^\nabla_{\mathring{\xi}} \mathring{a}
    = \mathring{\xi}\hook \dN(\mathring{a})
    = -\mathring{\xi}\hook\pi^*(\xi\hook \omega)
    = -\pi^*(\xi^*\omega)=0,
  \end{equation*}
  we conclude that $\mathring{a} = \pi_S^*a_S^{-1}$ for some vector
  bundle morphism $a_S\colon E_S\to F_S$.

  To define $\xi_S\colon E_S \to TS$, we wish to push
  $\rho\colon \pi_S^*E_S = \pi^*F\to TP$ forward so that the diagram
  \begin{equation*}
    \begin{CD}
      \pi_S^*E_S @>\rho>> TP\\
      @VVV @VV d\pi_S V\\
      E_S @>\xi_S>> TS.
    \end{CD}
  \end{equation*}
  commutes.  We thus define
  $(\xi_S)_x(\Se_x)\coloneqq (d\pi_S)_p(\rho_p(p,\Se_x))$ for any
  $x\in S$, $\Se_x\in (E_S)_x$ and for some $p\in \pi_S^{-1}(x)$, but
  need to show this is independent of~$p$. To this end, let $\Se$ be
  a local parallel section of~$E_S$ extending~$\Se_x$.  Then
  $\theta\rho(\pi_S^*\Se) = \pi_S^*\Se$ is parallel, so
  \eqref{eq:commuting-actions} gives that
  $[\rho(\pi_S^*\Se),\mathring\xi(\pi^*e)] = 0$ for each local
  parallel section $e$ of~$E$.  Thus $\rho(\pi_S^*\Se)$ is
  projectable, as required, and $\xi_S$ is well-defined.

  For the torsion-free condition, let $\Se_1$ and $\Se_2$ be
  two sections of~$E_S$.  Then $\rho(\pi_S^*\Se_i)$ projects to
  $\xi_S\Se_i$, so $[\rho(\pi_S^*\Se_1),\rho(\pi_S^*\Se_2)]$ projects
  to~$[\xi_S\Se_1,\xi_S\Se_2]$.  Moreover,
  \begin{equation*}
    \nabla_{\rho(\pi_S^*\Se_1)} \pi_S^*\Se_2 - \nabla_{\rho(\pi_S^*\Se_2)}
    \pi_S^*\Se_1
    = \pi_S^*\left(\nabla_{\xi_S(\Se_1)} \Se_2 - \nabla_{\xi_S(\Se_2)} \Se_1\right)
  \end{equation*}
  and so
  $\rho(\nabla_{\rho(\pi_S^*\Se_1)}
  \pi_S^*\Se_2-\nabla_{\rho(\pi_S^*\Se_2)}\pi_S^*\Se_1)$ projects to
  $\xi_S(\nabla_{\xi_S(\Se_1)} \Se_2-\nabla_{\xi_S(\Se_2)} \Se_1)$.
        Thus, $\xi_S$ is torsion-free as $\rho$ is torsion-free by
        Remark \ref{re:fibresaffinemfds}.

  Defining
  $\theta_S\coloneqq \mathring{a}^{-1}\circ \theta\in
  \Omega^1(P,\pi_S^*F_S)$ and $\rho_S \coloneqq \mathring\xi$, we
  verify the conditions~\ref{item:theta-rho}--\ref{item:d-theta} of
  Definition~\ref{def:total-space}.  Firstly,
  $\theta_S\circ\rho_S = \mathring a^{-1}\circ \theta \circ
  (\widetilde\xi+\rho\circ \mathring{a}) = \mathring a^{-1}\theta\rho
  \mathring a = \id_{\pi^*E} = \id_{\pi_S^*F_S}$.  Furthermore, $S$
  is defined to be the leaf space of $\mathring\xi$, so
  $d\pi_S\circ \mathring{\xi}=0$. Now a computation gives
  \begin{equation*}
    \begin{split}
      \dN\theta_S
      &= \dN\mathring{a}^{-1}\wedge\theta + \mathring{a}^{-1}\dN\theta
        = \pi^*(a^{-1}\xi\hook \omega a^{-1})\wedge \theta
        + \pi^*(a^{-1} \omega)\\
      &= \pi^*(a^{-1}\xi\hook \omega)\wedge \theta_S
        + \pi^*(a^{-1} \omega).
    \end{split}
  \end{equation*}
  From this we conclude
  $\mathring{\xi}\hook \dN\theta_S = \pi^*(a^{-1}\xi^*\omega)\wedge
  \theta_S = 0$ and so
  $\cL_{\mathring{\xi}}^{\nabla} \dN\theta_S = 0$. Thus, there is a
  two-form $\omega_S\in \Omega^2(S,F_S)$ with
  $\pi_S^*\omega_S = \dN\theta_S$.  Indeed
  $a^{-1}\omega \Hrel \omega_S$.

  Having found $a_S$ and~$\omega_S$, it remains to verify
  \ref{item:s-closure}--\ref{item:s-pullback} in
  Definition~\ref{def:shear}.
  \ref{item:s-closure}~is automatic from the flatness of~$\nabla$.
  We compute
  \begin{equation*}
    \begin{split}
      \pi_S^*\dN a_S
      &= \dN \mathring a^{-1}
        = - \mathring a^{-1} \dN \mathring a\, \mathring a^{-1}
        = \pi^*(a^{-1} \xi \hook \omega a^{-1})\\
      &= - \rho \hook \dN \theta_S
        = - \pi_S^* (\xi_S \hook \omega_S),
    \end{split}
  \end{equation*}
  giving~\ref{item:s-contract}.  The final condition follows from
  \begin{equation*}
    \pi_S^*(\xi_S^*\omega_S)
    = \rho^*(\pi_S^*\omega_S)
    = \rho^*\dN\theta_S
    = \rho^*\bigl(\pi^*(a^{-1}\xi\hook \omega)\wedge \theta_S
    + \pi^*(a^{-1} \omega)\bigr)
    = 0.
  \end{equation*}
  Thus we have some shear data $(\xi_S,\omega_S,a_S)$ on~$S$.

  To see that we obtain~$M$ as the resulting shear, it is
  sufficient to show that $\mathring \xi_S = \rho$. But
  $\mathring \xi = \widetilde \xi + \rho \circ \mathring a$ and
  $\pi_S^*a_S = \mathring a^{-1}$ imply
  \begin{equation*}
    \rho = - \widetilde \xi \circ \mathring a^{-1} + \mathring \xi
    \circ \mathring a^{-1}
    = - \widetilde\xi \circ \mathring a^{-1} + \rho_S \circ \pi_S^*a_S.
  \end{equation*}
  Now the final term is $\pi_S$-vertical and the first one
  lies in~$\cH = \ker \theta = \ker \theta_S$. As $\rho$
  projects to $\xi_S$, we conclude that
  $\widetilde{\xi_S} = - \widetilde \xi \circ \mathring a^{-1}$,
  thus $\mathring \xi_S = \rho$ and
  $- \xi \circ a^{-1} \Hrel \xi_S$.
\end{proof}

\section{Examples}
\label{sec:examples}
\subsection{Shears via non-principal bundles}
\label{subsec:shearsnonpb}
In \cite{FS}, we considered a first version of a rank one shear construction
with the bundles $E$ and $F$ trivialised.  We showed in \cite[\S3.4]{FS} that
then $P$ always has the structure of a principal bundle.  This is no
longer true in the general shear construction, even when $P$ has rank
one. This fails for two reasons: we no longer require $P$ to be a
fibre bundle and we only require $E$ and $F$ to be flat, not
trivial.

Let us give a simple example where $P$ is not a fibre bundle and
the topology of the fibres can change.  Take
$P = T^2 = S^1 \times S^1$, $M = S^1$ and $W = S^1$.
There is a twist of the form $M = S^1\leftarrow P = T^2\to W = S^1$
with the first map being the projection onto the first $S^1$-factor
and the second one given by $(x,y)\mapsto x^{-1}y$.  This corresponds
to twist data $(\xi,a,\omega)$ with Abelian Lie algebras
$\mfa_P = \mfa_M = \bR$, $\xi(1) = \partial_{\varphi}$,
$a = \id_{\bR}$, $\omega = 0$ and $\theta = d\psi$, where $\varphi$
describes the first $S^1$-factor and $\psi$ the second one in $T^2$.
Taking $E = F = M\times \bR$ with the natural flat connections, this
can also be understood as a shear construction.

Now we may remove one point $p\in T^2$.  Then, still
$M = S^1\leftarrow P' = T^2\setminus \{p\} \to S = S^1$ is a shear,
but neither map $P' \to M,S$ is a fibre bundle, indeed most
fibres are circles, but one fibre is homeomorphic to~$\bR$. More
generally, one may remove a closed segment~$L$ in one fibre of
$T^2 \to M = S^1$, then $P'' = T \setminus L \to S = S^1$ is
still surjective but has fibres which are topologically~$\bR$ over a
closed segment in~$S = S^1$.

Even if we assume that $P\to M$ is a fibre bundle, we cannot conclude
that $P$~has the structure of a principal bundle.  To see this, take a
non-trivial flat vector bundle $(E,\nabla)$ of rank~$k$ and
$P = F = E$.  Then $P = E \to M$ does not admit the structure of a
principal $\bR^k$-bundle: if it did, the structure group could be
reduced to the maximal compact subgroup $\{0\} \leqslant \bR^k$
contradicting non-triviality.

Take the Ehresmann connection
$\hat\theta\in \Omega^1(E,\cV)\subset \End(TE)$ associated
to~$\nabla$.
We get an induced splitting
$TE = \cH\oplus \cV\cong \pi^*TM\oplus \pi^*E$ and $\theta$
corresponds to the projection onto the second factor.
Using Proposition~\ref{pro:Ehresmannconnection}, we see that
$(P,\theta,\inc)$ is a shear total space for $\omega = 0$:
\begin{itemize}
\item For $X_i=\pi^*e_i$ vertical with $e_i\in \Gamma(E)$, $i=1,2$,
  one gets $[X_1,X_2]=0$ as
  $\varphi^{X_i}_t(\tilde{e})=\tilde{e}+te_i(\pi(\tilde{e}))$ is the flow of
  $X_i=\pi^*e_i$ and
  $\nabla_{X_i} X_j=\nabla_{X_i} \pi^*e_j=\pi^*(\nabla_{d\pi(X_i)}
  e_j)=0$ for all $\{i,j\}=\{1,2\}$. Thus, equation
  \eqref{eq:rhotorsionfree} holds.
\item Let $X=\pi^*e$ for a parallel (local) section $e\in \Gamma(E)$
  and let $Y$ be horizontal.
  Then $[X,Y]=\frac{d}{dt}|_{t=0} d\varphi^X_{-t}(Y(\varphi(t)))$
        for $\varphi^X_t(\tilde{e})=\tilde{e}+te(\pi(\tilde{e}))$.
  As our Ehresmann connection comes from a covariant derivative
  $\nabla$ on $E$, the differentials of the scalar multiplication and
  the addition on $E$ preserve the horizontal subbundle.
  Moreover, $de(d\pi(v))$ is horizontal for any $v\in TE$ as $e$ is
  parallel.
  So $d\varphi_{-t}(Y(\varphi(t)))$ is in the horizontal subbundle for
  all $t$, which implies that $X=\pi^*e$ preserves the horizontal
  subbundle.
\end{itemize}

Moreover, if $\xi\colon E\to TM$ is any bundle map and $a\colon E\to
E$ any bundle isomorphism, then $(\xi,a,0)$ defines shear data exactly
when $(\xi,\nabla)$ is torsion-free and $\nabla a = \dN a = 0$.  For
such a shear, we have $d\alpha\Hrel d_S\alpha_S$ if
$\alpha\Hrel\alpha_S$. As we may always take $\xi = 0$ and $a = \id_E$ to get a
shear $M\leftarrow E\to M$ for which both projections coincide, there are, in fact, examples of
rank one shears for which $P\rightarrow M$ does not have the structure of a principal
bundle.

Note that if $E = TM$ and $\xi = \id_{TM}$, then $(\xi,\nabla)$ being
torsion-free means exactly that $\nabla$ is torsion-free. We may take
then $a = J$ to be an almost complex structure on $M$. As $\nabla J =
0$ and $\nabla$ torsion-free implies that $J$ is integrable, a triple
$(J,\id_{TM},0)$ defines shear data if and only if $(M,J,\nabla)$ is a
special complex manifold in the sense of \cite{ACD} with $\nabla$
being complex, i.e.\ $\nabla J = 0$. If, conversely, we take $\xi=J$ being an almost
complex structure and $a=\id_{TM}$, then $(\id_{TM},J,0)$ defines
shear data precisely when $[JX,JY]=J(\nabla_{JX} Y-\nabla_{JY} X)$ for all
$X,Y\in \mathfrak{X}(M)$. If additionally $\nabla$ is torsion-free, this condition
implies that $J$ is integrable. Hence, for $\nabla$ being torsion-free, $(\id_{TM},J,0)$
defines shear data if and only if $(M,J,\nabla)$ is a
special complex manifold.

If $E=TM$, we may also take $\omega$ to be the torsion $T^{\nabla}$ of a
non-torsion-free flat connection $\nabla$.
Namely, first of all, a short computation shows
\begin{equation*} \dN T^{\nabla}(X,Y,Z) =
R^{\nabla}(X,Y)Z+R^{\nabla}(Y,Z)X+R^{\nabla}(Z,X)Y = 0.
\end{equation*} Furthermore, we have $\cV\cong \pi^*TM$ in $TP = TTM =
\pi^*TM\oplus \pi^*TM = \cH \oplus \cV$.  So the one-form $A\in
\Omega^1 (TM,\pi^*TM)$, $A \coloneqq \theta+\eta$ with $\theta$ as
above and $\eta\in \Omega^1(TM,\pi^*TM)$ being the projection onto the
horizontal subbundle, satisfies $\dN A = \dN \eta = \pi^*T^{\nabla}$.
Hence, $(\rho,A,T^{\nabla})$ is a shear total space for $T^{\nabla}$.
Then a triple $(\xi,a,T^{\nabla})$ with
$\xi,a\in \mathrm{End}(TM)$ and $a$ being invertible defines shear
data if and only if $T^{\nabla}|_{\xi(TM)\times \xi(TM)} = 0$,
$T^{\nabla}(\any,\xi(\any)) = \nabla a$ and $\nabla_{\xi} \xi$ is
symmetric.

\subsection{Shears on Lie algebras revisited}
\label{sec:shears-LAs-revisited}

In \S\ref{sec:shears-lie-algebras}, we introduced the shear
construction on Lie algebras by replacing in the left-invariant twist
construction central ideals and central extension by Abelian ones.
The results of that section then motivated our general definition of
shear data and the shear in Definition~\ref{def:shear}.  Here, we like
to see how we can recover the shear on Lie algebras from
\qq{left-invariant} shear data on a Lie group $G$.

\begin{definition}
  \label{def:leftinvsheardata} Let $G$ be a $1$-connected Lie group
        and $E$ and $F$ be trivial vector bundles of rank $k$ over $G$
        endowed with flat connections $\nabla^E$ and $\nabla^F$,
        respectively, which are both \emph{left-invariant} in the sense that
        the connection forms with respect to frames of constant sections are left-invariant.
        We write $E = G\times \mfa_G$ and $F = G\times \mfa_P$ for $k$-dimensional
  Abelian Lie algebras $\mfa_G$ and $\mfa_P$.

Suppose $(\xi,a,\omega)\in \Gamma(\Hom(E,TG))\times
\Gamma(\Hom(E,F))\times \Omega^2 (G,F)$ is shear data on~$G$. Let $G$
act by left translations and their differentials on $G$ and~$TG$,
respectively, on $E$ and $F$ by left-translation on the first and
trivially on the second factor and extend these actions naturally to tensor products of
these vector bundles. Then we say the shear data is
\emph{left-invariant} if $(\xi,a,\omega)$ are $G$-equivariant sections
of the corresponding bundles and $\cL_{\xi}^{\nabla^E} \alpha = 0$
for all left-invariant one-forms $\alpha\in \Omega^1 G$.
\end{definition}

Note that $(\xi,a,\omega)$ being $G$-equivariant is equivalent to
$\xi$ mapping constant sections of $E$ to left-invariant vector
fields, $a$~mapping constants sections of $E$ to constant sections
of~$F$ and $\omega$~mapping pairs of left-invariant vector fields to
constant sections of $F$. So we can consider $a$ as an element of
$\mfa_G^*\otimes \mfa_P$, $\xi$~as a Lie algebra homomorphism from
$\mfa_G$ to $\mfg$ and $\omega$~as an element of $\Lambda^2
\mfg^*\otimes \mfa_P$.  The flat connections are determined by
$\gamma\in \mfg^*\otimes \gl(\mfa_G)$ and $\eta\in \mfg^*\otimes
\gl(\mfa_P)$, so $\nabla_X^E e=X(e)+\gamma(X)(e)$ or $\nabla_X^F
e=X(f)+\eta(X)(f)$, for any $X\in \mathfrak{X}(G)$ and any $e
\in \Gamma(E)$, $f \in \Gamma(F)$.

Condition~\ref{item:s-contract} for shear data in
Definition~\ref{def:shear} is $\xi\hook \omega=-\dN a$.  As
\begin{equation*} (\dN a)(X,Y) = (\nabla_Y a)(X) = \nabla_Y^F (a X) -
a(\nabla_Y^E X) = \eta(Y)(a X) - a(\gamma(Y)(X))
\end{equation*} for all $X\in \mfa_P$, $Y\in \mfg$, this gives us
$\gamma=a^{-1}(\xi\hook \omega)+a^{-1} \eta a$, which is
equation~\eqref{eq:gamma}.  Moreover, $(\xi,\nabla)$ being
torsion-free yields
\begin{equation}
  \label{eq:li-torsion-free} [\xi X,\xi Y] = \xi(\nabla_{\xi X
}Y-\nabla_{\xi Y}X) = \xi(\gamma(\xi X)(Y)-\gamma(\xi Y)(X))
\end{equation} for all $X,Y\in \mfa_G$.  Furthermore, we have
\begin{equation*}
  \begin{split} 0&= (\cL_{\xi}^{\nabla} \alpha)(X ,\xi Y) = \dN
(\xi\hook \alpha)(X,\xi Y)+(\xi\hook d\alpha)(X,\xi Y)\\
&=(\nabla_{\xi Y} \alpha\circ \xi)(X) +d\alpha(\xi X,\xi Y ) =
-\alpha(\xi \nabla_{\xi Y} X+[\xi X,\xi Y])\\ &=
-\alpha(\xi(\gamma(\xi Y)(X))+[\xi X,\xi Y])
  \end{split}
\end{equation*} for all $\alpha\in \mfg^*$, and so $[\xi X,\xi Y] =
-\xi(\gamma(\xi Y)(X))$ for all $X,Y\in \mfa_G$. With
\eqref{eq:li-torsion-free} we deduce that $\xi(\gamma(\xi X)(Y))=0$
for all $X,Y\in \mfa_G$ and so $[\xi X,\xi Y]=0$. Hence
$\xi\colon\mfa_G\to \mfg$ is a Lie algebra homomorphism. Thus, we
got the same data as in \S\ref{sec:shears-lie-algebras}
fulfilling the requirements of Lemma~\ref{lem:lift-cond}, except that
$\xi$ is not necessarily injective.

Definition~\ref{def:shear}\ref{item:s-closure} yields $0 = \dN
\omega = d\omega+\eta\wedge\omega$. Using the flatness of~$\nabla^F$,
we get that $\eta\in \mfg^*\otimes \gl(\mfa_P) =
\Hom(\mfg,\gl(\mfa_P))$ is a representation.

We may define now a natural shear total space for $\omega$ as follows:
First of all, as in
\S\ref{sec:shears-lie-algebras}, $\omega$ and $\eta$ define an Abelian
extension $\mfa_P\hookrightarrow\mfp\twoheadrightarrow \mfg$ together
with a vector space splitting $\mfp = \mfg\oplus \mfa_P$. Let $P$ be
the $1$-connected Lie group with Lie algebra $\mfp$ and define
$\theta\in \Omega^1(P,\pi^*F) = \Omega^1(P,\mfa_P)$ and
$\rho\colon P\times \mfa_P = \pi^*F\to TP$ as in \S\ref{sec:shears-lie-algebras},
i.e. $\theta$ corresponds to the projection $\mfp\rightarrow\mfa_P$ onto $\mfa_P$
induced by the splitting and $\rho$ to the injection $\mfa_P\hookrightarrow\mfp$.
Then $(P,\theta,\rho)$ is a shear total space for $\omega$.

Moreover, $\mathring{\xi}\colon P\times \mfa_G = \pi^*E\to TP$, $\mathring{\xi}
= \tilde{\xi}+\rho\circ a$ can be considered as a linear map from
$\mfa_G$ to $\mfp$ and Lemma~\ref{lem:lift-cond} gives us that $\mathring{\xi}$ is a Lie
algebra homomorphism and that $\mathring{\xi}(\mfa_G)$ is an Abelian
ideal in~$\mfp$. The leaves of the distribution
$\mathring{\xi}(\mfa_G)$ are the left cosets of the normal Lie
subgroup $N$ of~$P$ with Lie algebra $\mathring{\xi}(\mfa_G)$. Thus,
the shear is the Lie group $H \coloneqq P/N$ with Lie algebra $\mfh$
agreeing with that of Definition~\ref{def:shear-LAs}. In this sense,
the left-invariant shear on $1$-connected Lie groups with injective $\xi$
via the natural shear total space $(P,\theta,\rho)$ from above and the shear
on Lie algebras presented in \S\ref{sec:shears-lie-algebras} are the same and
we will not distinguish them in the following sections.

\subsection{Shears on almost Abelian Lie algebras}
\label{subsec:shearsaALAs} As explained at the end of
\S\ref{sec:shears-lie-algebras}, we can successively shear $\bR^n$ to
any $n$-dimensional solvable Lie algebra such that each shear
increases the solvable step length by one. One important example of solvable
Lie algebras of step length one is provided by almost Abelian Lie
algebras: $\mfg$~with a codimension one Abelian ideal~$\mfu$.
Choosing $X\in \mfg\backslash \mfu$, these Lie algebras are
determined by a single endomorphism $f\coloneqq \ad(X)|_{\mfu}
\in \End(\mfu)$. Experience shows that a range of different geometric
structures may be constructed on these~$\mfg$ \cite{F1,F2}.

The following proposition gives conditions when left-invariant data
$(\xi,a,\omega)\in \Hom(\mfa_G,\mfg)\times \Hom(\mfa_G,\mfa_P)\times
\Lambda^2 \mfg^*\otimes \mfa_P$ of a particular form is shear data on
the associated simply-connected Lie group $G$ and when the shear of a
closed left-invariant form is again closed.  We write $\alpha\in
\mfg^*$ for the unique element in the annihilator of~$\mfu$ with
$\alpha(X) = 1$.

We will consider an $f$-invariant subspace~$\mfa$ of~$\mfu$, take
$\mfa_G=\mfa=\mfa_P$ and let $E \coloneqq G\times \mfa_G$, $F
\coloneqq G\times \mfa_P$ be trivial vector bundles endowed with flat
left-invariant connections determined by $\gamma\in \mfg^*\otimes
\gl(\mfa_G)$ and $\eta \in \mfg^*\otimes \gl(\mfa_P)$, respectively.

\begin{proposition}\label{pro:sheardata-AALAs} Let $G$ be a
simply-connected almost Abelian Lie group with Lie algebra data
$(\mfg,\mfu,X,\alpha,f)$.  Let $\mfa \subset \mfu$ be an $f$-invariant
subspace and take flat bundles as above.

  Fix a left-invariant two-form $\omega \in \Lambda^2 \mfg^*\otimes
\mfa_P$.  Consider the decomposition $\omega = \omega_0+\alpha\wedge
\nu$, with $\omega_0\in \Lambda^2\mfu^*\otimes \mfa$ and $\nu\in
\mfu^*\otimes \mfa\subset\End(\mfu)$.

  Assume that $\omega_0\neq 0$ and $\omega_0|_{\mfa\otimes \mfu}=0$.
  \begin{enumerate}[\upshape(a)]
  \item\label{item:aa-shear} Then $(\inc,\id_{\mfa},\omega)$, with
$\inc\colon \mfa \to \mfg$ the inclusion, is left-invariant shear data
on~$G$ if and only if
    \begin{equation}
      \label{eq:sheareq1} f.\omega_0 = -(f+\nu) \circ \omega_0, \quad
\gamma = \alpha\otimes f|_{\mfa}, \quad \eta = \alpha \otimes
(f+\nu)|_{\mfa}.
    \end{equation}
  \item\label{item:aa-closed} Suppose $(\inc,\id_{\mfa},\omega)$ is
left-invariant shear data on $G$.  Let $\psi\in \Lambda^r \mfg^*$ be a
closed left-invariant $r$-form on~$G$ and decompose $\psi$ uniquely as
$\psi = \chi \wedge \alpha + \tau$ with $\chi \in \Lambda^{r-1}\mfu^*$
and $\tau\in \Lambda^r \mfu^*$.  Then the $\cH$-related form
$\psi_{\mfh}$ on the shear~$\mfh$ is closed if and only if
    \begin{equation}
      \label{eq:sheareq2} \kappa(\tau\wedge \omega_0)=0, \quad
\kappa(\chi \wedge \omega_0) + (-1)^{r-1}\nu.\tau=0,
    \end{equation} where $\kappa\colon \Lambda^k \mfu^*\otimes \mfa
\to \Lambda^{k-1}\mfu^*$ is the unique linear map given on
decomposable elements $\rho\otimes A$ by $\kappa(\rho\otimes A) =
A\hook \rho$.
  \end{enumerate}
\end{proposition}

In the above statement $(f.\omega_0)(X,Y) = - \omega_0(fX,Y) -
\omega_0(X,fY)$ for $X,Y \in \mfu$.

\begin{remark}\label{re:specialchoice} Here is one motivation for the
case considered. As $a\colon\mfa_G\to \mfa_P$ is invertible, it is no
restriction to assume that $\mfa = \mfa_G = \mfa_P$ and that $a =
\id_{\mfa}$. Moreover, it is natural to assume that $\xi\colon\mfa_G \to
\mfg$ is injective and so we may take $\xi$ to be the
inclusion. The image of $\xi$ is an Abelian ideal, so it is reasonable to take
$\mfa \subset \mfu$ as $f$-invariant subspaces of $\mfu$ are such
ideals.

  The condition $\omega_0\neq 0$ is equivalent to the new algebra
$\mfh$ obtained by the shear no longer being almost Abelian.  Finally,
for $(\inc, \id|_{\mfa},\omega)$ to be shear data, we must have
$\xi^*\omega=0$, which is equivalent to $\omega_0|_{\Lambda^2
\mfa}=0$. This is ensured by the stronger condition
$\omega_0|_{\mfa\otimes \mfu}=0$, which is much simpler to work with.
\end{remark}

\begin{proof}[Proof of Proposition~\ref{pro:sheardata-AALAs}]
  \ref{item:aa-shear} By Remark~\ref{re:specialchoice}, condition
  \ref{item:s-pullback} in Definition~\ref{def:shear} is fulfilled
  since $\omega_0|_{\mfu\otimes \mfa} = 0$.
  Moreover, the discussion in~\S\ref{sec:shears-LAs-revisited} gives
  us that the validity of \ref{item:s-contract} in
  Definition~\ref{def:shear} for $(\inc,\id_{\mfa},\omega)$ is
  equivalent to
  $\gamma = \omega|_{\mfa\otimes \mfg}+\eta = -\alpha\otimes
  \nu|_{\mfa} + \eta$.
  As in \S\ref{sec:shears-lie-algebras}, we see that
  $\cL_{\xi}^\nabla \beta = 0$ for all $\beta\in \mfg^*$ is equivalent
  to $\gamma(X) = \xi\circ \gamma(X) = [X,\xi(\any)] = \ad(X)|_{\mfa}$
  for all $X\in \mfg$, and so to $\gamma = \alpha\otimes f|_{\mfa}$.
  Then $\eta = \alpha \otimes (f+\nu)|_{\mfa}$. The formulas for
  $\eta$ and $\gamma$ imply that the associated connections are flat
  and that $(\inc, \nabla)$ is torsion-free.

  Finally, we have to check when $\dN \omega = 0$ holds.  We have
  \begin{equation*} 0 = \dN \omega = d\omega + \eta\wedge \omega =
    \alpha \wedge f.\omega_0 + \alpha \wedge (f+\nu) \circ \omega_0.
  \end{equation*} So $\dN \omega = 0$ if and only if $f.\omega_0 =
  -(f+\nu)\circ \omega_0$.

  \ref{item:aa-closed} By the formula for the differential
  $d_{\mfh}\psi$ in Corollary~\ref{cor:differentials} and since
  $d\psi = 0$, we have to investigate when
  $(\xi\hook \psi) \wedge \omega$ is zero. First of all,
  \begin{equation*}
    \begin{split}
      (\xi\hook \psi) \wedge \omega
      &= (\xi\hook \chi)
        \wedge \alpha \wedge \omega_0 + (\xi\hook \tau) \wedge \alpha \wedge
        \nu + (\xi\hook \tau)\wedge \omega_0\\
      &= \bigl( (\xi\hook \chi)
        \wedge \omega_0 - (\xi\hook \tau) \wedge \nu \bigr) \wedge \alpha
        +(\xi\hook \tau)
        \wedge \omega_0\\
      &= \bigl( \kappa(\chi \wedge
        \omega_0) - (\xi\hook \tau)\wedge \nu \bigr) \wedge \alpha +
        \kappa(\tau\wedge \omega_0),
    \end{split}
  \end{equation*} since $\mfa$ is in the kernel of~$\omega_0$.  Now
  also
  \begin{equation*}
    \begin{split}
      -\bigl((\xi\hook \tau)\wedge
      \nu\bigr)(X_1,\dots,X_r)
      &=-\sum_{i=1}^r (-1)^{r-i}
        \tau(\nu(X_i),X_1,\dots,\widehat{X_i},\dots,
        X_r)\\
      &=-(-1)^{r-1}
        \sum_{i=1}^r \tau(X_1,\dots, \nu(X_i),\dots, X_r)\\
      &=(-1)^{r-1}(\nu.\tau)(X_1,\dots,X_r)
    \end{split}
  \end{equation*} for all $X_1,\dots, X_r\in \mfg$, and the result
  follows.
\end{proof}

\subsubsection{Cocalibrated $\G_2$-structures}

Suppose we have a $\G_2$-structure $\varphi\in \Lambda^3 \mfg^*$ on a
seven-dimensional almost Abelian Lie algebra $\mfg$ with Abelian ideal
$\mfu$ of codimension~$1$. Choosing some unit-length $\alpha\in \mfg^*$ in the
annihilator of~$\mfu$, it is well-known, cf.\ e.g.~\cite{MC}, that
$\varphi$ naturally induces a special almost Hermitian structure
$(\sigma,\rho)\in \Lambda^2 \mfu^*\times \Lambda^3 \mfu^*$ on~$\mfu$
with
\begin{equation*} \Hodge_{\varphi}\varphi = \tfrac12\sigma^2 +
\rho\wedge \alpha.
\end{equation*} We use the convention that the two-form~$\sigma$, the
induced almost complex structure $J$ and the induced Riemannian
metric~$g$ are related by $\sigma = g(\any,J\any)$.  Put $X\in
\mfg\backslash \mfu$ to be the unique element orthogonal to $\mfu$
with $\alpha(X) = 1$.

Suppose $\varphi$ is \emph{cocalibrated} $d\Hodge_{\varphi}\varphi =
0$.  Then $f \coloneqq \ad(X)|_{\mfu}$ lies in $\sP(\mfu,\sigma)$
by~\cite{F1}, so $\tr(f) = 0$.  Consider $\omega = \omega_0 + \alpha
\wedge \nu\in \Lambda^2\mfg^* \otimes \mfa$ fulfilling the
requirements of Proposition~\ref{pro:sheardata-AALAs} with respect to
some $\mfa \subset \mfu$.  We aim at a partial classification of all
such $\omega$ for which $(\inc,\id_{\mfa},\omega)$ is left-invariant
shear data with the shear $\varphi_\mfh$ of $\varphi$ cocalibrated.

First note that the $\gl(\mfa)$-valued left-invariant one forms $\eta$
and $\gamma$ which define the flat connections are fixed by
equation~\eqref{eq:sheareq1}.  Furthermore, as $\mfu$ is
six-dimensional, $\kappa\colon\Lambda^6\mfu^* \otimes \mfa \to
\Lambda^5 \mfu^*$ is injective.  It follows that the first equation in
\eqref{eq:sheareq2} is given by $\sigma^2\wedge \omega_0 = 0$, which
in turn is equivalent to
\begin{equation*} \omega_0\in \sigma^\bot\otimes \mfa =
[\Lambda_0^{1,1}\mfu^*]\otimes \mfa + \llbracket\Lambda^{2,0}
\mfu^*\rrbracket\otimes \mfa.
\end{equation*} So by Proposition~\ref{pro:sheardata-AALAs} we are
left with solving the second equation in \eqref{eq:sheareq2} and the
first one in \eqref{eq:sheareq1}.  This is complicated, so we restrict
to certain special cases and obtain the following result.

\begin{proposition}\label{pro:shearcocG2} Let $(\mfg,\mfu,\varphi)$ be
a seven-dimensional almost Abelian Lie algebra with a cocalibrated
$\G_2$-structure $\varphi\in \Lambda^3 \mfg^*$.  Let $\mfa$ be an
$f$-invariant subspace of $\mfu$ and $\omega = \omega_0+\alpha\wedge
\nu\in \Lambda^2 \mfg^*\otimes \mfa$ be as in
Proposition~\ref{pro:sheardata-AALAs}.
  \begin{enumerate}[\upshape(a)]
  \item\label{item:omega} If $\omega_0\in
[\Lambda_0^{1,1}\mfu^*]\otimes \mfa$, then
$(\mathrm{inc},\id_{\mfa},\omega)$ shears $(\mfg,\varphi)$ to a
cocalibrated structure if and only if $\dim(\mfa)\leq 2$, $\nu\in
\lie{sp}(\mfu,\sigma)$ and either
    \begin{enumerate}[\upshape(i)]
    \item\label{item:omega-1} $\dim(\im(\omega_0)) = 1$, $f.\omega_0 =
0$ and $\nu|_{\im(\omega_0)} = -f|_{\im(\omega_0)}$, or
    \item\label{item:dim2} $\dim(\im(\omega_0)) = 2$, $\omega_0 =
\sum_{i = 1}^2 \tilde{\omega}_i\otimes Y_i$, with $Y_1,Y_2\in \mfa$ a
basis, $\tilde{\omega}_i\wedge \tilde{\omega}_j =
\delta_{ij}\tilde{\omega}_1^2$ for $i,j\in \{1,2\}$,
$f.\tilde{\omega}_1 = a \, \tilde{\omega}_2$, $f.\tilde{\omega}_2 = -a
\, \tilde{\omega}_1$, $\nu(Y_1) = a Y_2-f(Y_1)$ and $\nu(Y_2) = -a
Y_1-f(Y_2)$ for some $a\in \bR$.
    \end{enumerate}
  \item\label{item:dim4firstcase} If $\dim(\mfa) = 4$ and
$J(\im(\omega_0))\perp \mfa$, then $(\mathrm{inc},\id_{\mfa},\omega)$
shears~$(\mfg,\varphi)$ to a cocalibrated structure if and only if
$f|_{\im(\omega_0)} = -\tr(f|_{\mfa})\id_{\im(\omega_0)}$ and
    \begin{equation*} \nu\in \tilde{\nu}+\bigl\{\,\hat{\nu}\in
\lie{sp}(\mfu,\sigma) \bigm| \hat{\nu}(\mfu)\subset \mfa,\
\hat{\nu}|_{\im(\omega_0)} = 0\,\bigr\},
    \end{equation*} where $\tilde{\nu}\in \End(\mfu)$ is
$\tilde{\nu}(W) = -\rho(JW,\kappa(J\circ
\omega_0)^{\sharp},\any)^{\sharp}$ for $W\in J\im \omega_0$ and
$\tilde{\nu}|_{(J\im \omega_0)^{\perp}} = 0$.
  \item\label{item:dim4secondcase} If $\dim(\mfa) = 4$ and
$J(\im(\omega_0))\subset\mfa$, then $(\mathrm{inc},\id_{\mfa},\omega)$
shears $(\mfg,\varphi)$ to a cocalibrated structure if and only if
$\mfa$ is a $\sigma$-degenerate subspace of $\mfu$ and
    \begin{equation*} \nu\in \tilde{\nu}+\bigl\{\,\hat{\nu}\in
\lie{sp}(\mfu,\sigma) \bigm| \hat{\nu}(\mfu)\subset \mfa,\
\hat{\nu}|_{\im(\omega_0)} = 0\,\bigr\},
    \end{equation*} where $\tilde{\nu}$ is given for $Y \in
\im(\omega_0)$ with $\norm{ Y } = 1$ by
    \begin{gather*} \tilde{\nu}(Y) = -\tr(f|_{\mfa})Y-f(Y),\\
\tilde{\nu}(JY) = \left(-\sigma(f(Y),JY)+\tr(f|_{\mfa})-\mu\right)JY,
\\ \tilde{\nu}|_{\mfa^{\perp}} = -\sigma(f(Y),\any)|_{\mfa^{\perp}}
JY,
    \end{gather*} and $\tilde{\nu}|_U = \mu\id_{U}$ on $U =
(\im\omega_0 + J\im \omega_0)^{\perp}\cap \mfa$, with $\mu\in \bR$
fixed by $\kappa(\omega_0\wedge \rho)|_{\Lambda^4\spa{Y,JY}^{\perp}} =
-\mu \sigma^2|_{\Lambda^4 \spa{Y,JY}^{\perp}}$.
  \end{enumerate}
\end{proposition}

\begin{proof}
  \ref{item:omega} Here, $\rho\wedge \omega_0 = 0$ and so the second
  equation in \eqref{eq:sheareq2} simplifies to
  $\nu.\sigma \wedge \sigma = 0$, i.e.\ to $\nu.\sigma = 0$, since the
  Lefschetz operator is bijective on two-forms in six dimensions.
  Thus we regard $\nu$ as an endomorphism of~$\mfu$ and see that
  $\nu\in \sP(\mfu,\sigma)$.

  Let $\Hodge_\mfu$ be the Hodge star operator on~$\mfu$. Then using
  Schur's Lemma and a concrete element, we find
  $\Hodge_\mfu \tilde{\omega} = -\tilde{\omega}\wedge \sigma$ for each
  $\tilde{\omega}\in [\Lambda^{1,1}_0]$.
  Hence, for such $\tilde\omega$,
  $\sigma \wedge \tilde{\omega}^2 = -\tilde{\omega}\wedge
  \Hodge_\mfu\tilde{\omega} =
  -g(\tilde{\omega},\tilde{\omega})\tfrac16\sigma^3$, showing that
  $\sigma \wedge \tilde{\omega}^2$ is non-zero if $\tilde{\omega}$ is
  non-zero.
  In particular, any non-zero element in $[\Lambda_0^{1,1}]$ has rank
  at least four.
  The condition $\omega_0|_{\mfa\otimes\mfu} = 0$, then gives
  $\dim(\mfa)\leq 2$, and so $\dim(\im(\omega_0))\leq 2$.

  If $\dim(\im(\omega_0)) = 1$, then
  $(f+\nu)\circ\omega_0 = \lambda\omega_0$ for some $\lambda\in\bR$,
  and the first equation in~\eqref{eq:sheareq1} gives
  $f.\omega_0 = -\lambda \omega_0$.
  Now
  $0\neq \sigma\wedge \omega_0^2\in \Lambda^6 \mfu^*\otimes
  \mfa^{\otimes 2}$ and so, recalling that $\tr(f) = 0$, we have
  \begin{equation*} 0 = \tr(f)\sigma\wedge \omega_0^2 =
    f.(\sigma\wedge \omega_0^2) = 2\sigma\wedge f.\omega_0 \wedge
    \omega_0 = -2\lambda \sigma\wedge \omega_0^2,
  \end{equation*} giving $\lambda = 0$.  Thus, $f.\omega_0 = 0$ and
  $\nu|_{\im(\omega_0)} = -f|_{\im(\omega_0)}$.

  Let us now consider the case $\dim(\im(\omega_0)) = 2$.
  Then $\dim(\mfa) = 2$ and $\mfa$~is a $J$-invariant subspace as it
  is the kernel of a $(1,1)$-form.
  It follows that $\tr(f|_{\mfa}) = 0$.
  As the space of four-forms on $\mfu$ with annihilator~$\mfa$ is
  one-dimensional and the square of any non-zero element in
  $[\Lambda_0^{1,1}\mfu^*]$ which annihilates $\mfa$ is non-zero, we
  may choose a basis $Y_1,Y_2$ of~$\mfa$, so that
  $\omega_0 = \tilde{\omega}_1\otimes Y_1 + \tilde{\omega}_2\otimes
  Y_2$ with
  $\tilde{\omega}_1,\tilde{\omega}_2\in [\Lambda_0^{1,1}\mfu^*]$
  annihilating $\mfa$ and satisfying
  $\tilde{\omega}_i \wedge \tilde{\omega}_j = \delta_{ij}
  \tilde{\omega}_1^2 \neq 0$.
  Since $\tr(f|_{\mfa}) = 0$ and $\tr(f) = 0$, we have
  $f.(\tilde{\omega}_i\wedge \tilde{\omega}_j) = 0$ for all
  $i,j = 1,2$.
  Moreover, the first equation in \eqref{eq:sheareq1} yields
  $f.\tilde{\omega}_i \in \spa{\tilde{\omega}_1,\tilde{\omega}_2}$ for
  $i = 1,2$.
  Hence, we obtain $f.\tilde{\omega}_1 = a\, \tilde{\omega}_2$ and
  $f.\tilde{\omega}_2 = -a\, \tilde{\omega}_1$ for some $a\in \bR$.
  But then the first equation in \eqref{eq:sheareq1} is equivalent to
  $\nu(Y_1) = a Y_2-f(Y_1)$ and $\nu(Y_2) = -a Y_1-f(Y_2)$.

  \ref{item:dim4firstcase}~\&~\ref{item:dim4secondcase} For these
  cases, note that $\omega_0$ has kernel equal to~$\mfa$, that
  $\omega_0^2 = 0$ and that $\dim(\im\omega_0) = 1$.
  So the equation $\sigma^2\wedge \omega_0 = 0$ is equivalent to
  $\mfa$~being a $\sigma$-degenerate subspace, as claimed.

  The $f$-invariance of $\mfa$ and $\tr(f) = 0$ give
  $f.\omega_0 = \tr(f|_{\mfa}) \omega_0$.
  So the first equation in \eqref{eq:sheareq1} is equivalent to
  $\nu|_{\im(\omega_0)} =
  -\tr(f|_{\mfa})\id_{\im(\omega_0)}-f|_{\im(\omega_0)}$ and we are
  left with solving the equation
  \begin{equation}\label{eq:remainingequation}
    \kappa\left(\omega_0\wedge \rho\right) = \nu.\sigma\wedge \sigma.
  \end{equation} Note that the space of $\nu\colon\mfu\to \mfa\subset
  \mfu$ solving equation~\eqref{eq:remainingequation} and with
  $\nu|_{\im(\omega_0)}$ given as above is an affine subspace of
  $\End(\mfu)$ modelled on $\{\,\hat{\nu}\in
  \lie{sp}(\mfu,\sigma)\mid\hat{\nu}(\mfu)\subset \mfa,\
  \hat{\nu}|_{\im\omega_0} = 0\,\}$.

  \ref{item:dim4firstcase} Before checking that $\tilde{\nu}$ as in
  the statement is a solution of
  equation~\eqref{eq:remainingequation}, we show that any solution
  $\nu$ of equation~\eqref{eq:remainingequation} has to fulfil
  $\nu|_{\im\omega_0} = 0$ and so we must have
  $f|_{\im\omega_0} = -\tr(f|_{\mfa})\id_{\im\omega_0}$.
  As $\dim(\im(\omega_0)) = 1$, $\im(\omega_0)$ lies in the kernel of
  $\kappa(\omega_0\wedge \rho)$.
  Let $Y$ be a non-zero element of $\im(\omega_0)$.
  Then $\omega_0 = \tilde{\omega}\otimes Y$, $\tilde{\omega}$ has
  kernel~$\mfa$ and
  $\kappa(\omega_0\wedge\rho) = \tilde{\omega}\wedge (Y\hook \rho)$.
  Hence,
  $0 = Y \hook \kappa(\omega_0\wedge\rho) = (Y\hook \nu.\sigma) \wedge
  \sigma + \nu.\sigma\wedge (Y\hook \sigma)$.  We define
  \begin{equation*}
    Z = \spa{Y,JY}^\perp \subset \mfu.
  \end{equation*}
  Restricting the previous identity to $\Lambda^3Z$, we get
  $Y\hook \nu.\sigma|_Z = 0$. As
  $\im(\nu)\subset \mfa \subset (JY)^{\perp}$, this implies
  $\nu(Y)\in \spa{Y,JY}\cap \mfa=\spa{Y}$. So $\nu(Y) = \lambda Y$ for
  some $\lambda\in \bR$.
  Then $0= (\nu.\sigma\wedge \sigma)(Y,JY,\any,\any)|_{\Lambda^2Z}$
  gives us $\nu.\sigma = \lambda\sigma$ on ${\Lambda^2 Z}$. Now the
  kernel of the four-form $\nu.\sigma\wedge \sigma$ has to be at least
  two-dimensional. Moreover,
  $(\nu.\sigma\wedge \sigma)|_{\Lambda^2 \mfa^{\perp}\wedge \Lambda^2
  \mfa}=\tilde{\omega}|_{\Lambda^2 \mfa^{\perp}}\wedge (Y\hook
  \rho)|_{\Lambda^2 \mfa}\neq 0$ and
  $(\nu.\sigma\wedge \sigma)|_{ \mfa^{\perp}\wedge \Lambda^3
  \mfa}=0$. So the kernel of $\nu.\sigma\wedge \sigma$ is contained
  in~$\mfa$ and, hence, has non-zero intersection with~$Z$. Thus,
  $0=(\nu.\sigma\wedge)|_{\Lambda^4 Z}=\lambda \sigma^2|_{\Lambda^4
  Z}$ giving $\lambda = 0$ as claimed.

  Now consider $\tilde{\nu}$. Let $W\in J\im(\omega_0)$ be non-zero.
  We first note that $\tilde{\nu}(W)\in \mfa$, as $\mfa$ is in the
  kernel of $\omega_0$, so
  $\kappa(J\circ \omega_0)^{\sharp}\in (\mfa\oplus
  J\im\omega_0)^{\perp}$, which gives
  $\tilde{\nu}(W) = -\rho(JW,\kappa(J\circ
  \omega_0)^{\sharp},\any)^{\sharp}\in \spa{W,\kappa(J\circ
  \omega_0)^{\sharp}}^{\perp} = \mfa$. Moreover, both
  $\kappa(\omega_0\wedge \rho)$ and $\tilde{\nu}.\sigma\wedge \sigma$
  are zero when we restrict to $\Lambda^4 (J\im(\omega_0))^{\perp}$.

  Besides,
  $(\sigma\wedge Y\hook\rho)|_{\Lambda^4 \spa{JY}^{\perp}} = 0$.
  Since
  $J(JY\hook \tilde{\omega})^{\sharp} = J\kappa(J\circ
  \omega_0)^{\sharp}\in \linebreak\spa{JY}^{\perp}$, we get on
  $\Lambda^3 (J\im\omega_0)^{\perp} = \Lambda^3 \spa{JY}^{\perp}$ that
  \begin{equation*}
    \begin{split} JY\hook (\tilde{\nu}.\sigma\wedge \sigma) &=
                                                              -\sigma(\tilde{\nu}(JY),\any)\wedge \sigma = -g(\rho(Y,\kappa(J\circ
                                                              \omega_0)^{\sharp},\any)^{\sharp},J\any)\wedge \sigma\\ &=
                                                                                                                        -\sigma\wedge \rho(Y,J(JY\hook \tilde{\omega})^{\sharp},\any) =
                                                                                                                        \sigma(J(JY\hook \tilde{\omega})^{\sharp},\any)\wedge (Y\hook \rho)\\
                                                            &= g((JY\hook \tilde{\omega})^{\sharp},\any)\wedge (Y\hook \rho) =
                                                              (JY\hook \tilde{\omega})\wedge (Y\hook \rho)= JY\hook
                                                              \kappa(\omega_0\wedge \rho),
    \end{split}
  \end{equation*} so \eqref{eq:remainingequation} is satisfied.

  \ref{item:dim4secondcase} Using the results from above, we only have
  to show that $\tilde{\nu}$ defined as in
  Proposition~\ref{pro:shearcocG2}\ref{item:dim4secondcase} solves
  equation~\eqref{eq:remainingequation}. Now, as $JY\in \mfa$, the
  left-hand side is zero if we insert $Y$ or~$JY$. Since
  $\mfa^{\perp} = JU$, straightforward computations show
  \begin{gather*} \tilde{\nu}.\sigma(Y,JY) = \mu \sigma(Y,JY),\qquad
    \tilde{\nu}.\sigma|_{\Lambda^2 U} = 0 =
    \tilde{\nu}.\sigma|_{\Lambda^2 JU},\\
    \tilde{\nu}.\sigma|_{\spa{Y,JY}\wedge (U\oplus JU)} = 0,\qquad
    \tilde{\nu}.\sigma|_{U\wedge JU} = -\mu \sigma|_{U\wedge JU}.
  \end{gather*} So $\tilde{\nu}.\sigma\wedge \sigma$ is also zero if
  we insert $Y$ or~$JY$.
  Finally, on $\Lambda^4(U\oplus JU)$ we have
  $\tilde{\nu}.\sigma\wedge \sigma = -\mu \sigma^2 =
  \kappa(\omega_0\wedge \rho)$ as required, since
  $U\oplus JU = \spa{Y,JY}^{\perp}$.
\end{proof}

Let us give examples of all cases in Proposition~\ref{pro:shearcocG2}.

\begin{example}\label{ex:cocalibratedshears} Look at the almost
Abelian Lie algebra defined by
  \begin{equation*} (a_1.17,a_2.27,a_3.37,-a_1.47,-a_2.57,-a_3.67,0)
  \end{equation*} for $a_1,a_2,a_3\in \bR$, where $a_1.17$ in
place~$1$ means that $de^1 = a_1 e^{17}$ with respect to the basis
$e^1,\dots,e^7$ of $\mfg^*$, etc.  Consider the cocalibrated
$\G_2$-structure $\varphi\in \Lambda^3 \mfg^*$ with closed Hodge dual
$\Hodge_{\varphi}\varphi = 1425+1436+2536+1237-1567+2467-3457$, where
$1425 \coloneqq e^{1425} \coloneqq e^1\wedge e^4\wedge e^2\wedge e^5$,
etc.

  Case~\ref{item:omega}\ref{item:omega-1}: Taking $\mfa =
\spa{e_1,e_4}$, $\omega_0 = (e^{36}-e^{25})\otimes e_1$ and $\nu(e_1)
= -a_1 e_1$, $\nu(e_4) = a_1 e_4$ and $\nu(e_i) = 0$ for $i\in
\{2,3,5,6\}$, the shear gives a cocalibrated $\G_2$-structure on
  \begin{equation*} (25-36, a_2.27, a_3.37,0, -a_2.57, -a_3.67, 0).
  \end{equation*}

  Case~\ref{item:dim4firstcase}: Assume that $a_3=-2a_1$.
Then, taking $\mfa = \spa{e_1,e_2,e_3,e_5}$,
$\omega_0 = -e^{46}\otimes e_1$ and $\nu(e_4) = -e_5$, $\nu(e_i) = 0$
for all $i\in \{1,2,3,5,6\}$, the shear gives a cocalibrated
$\G_2$-structure on
  \begin{equation*} \bigl(a_1.17+46, a_2.27, -2 a_1.37,
-a_1.47, -47-a_2.57, 2a_1.67, 0\bigr).
  \end{equation*}

  Case~\ref{item:dim4secondcase}: Taking $\mfa =
\spa{e_1,e_4,e_5,e_6}$, $\omega_0 = -ce^{23}\otimes e_1$ for some
$c\in \bR$ and $\nu(e_1) = (a_2+a_3-a_1)e_1$, $\nu(e_4) =
(a_1-a_2-a_3-\tfrac{c}{2})e_4$, $\nu(e_i) = \tfrac{c}{2} e_i$ for $i =
5,6$ and $\nu(e_j) = 0$ for $j = 2,3$ and setting $b = a_2+a_3$, the
shear gives a cocalibrated $\G_2$-structure on
  \begin{equation*} \Bigl(b.17+c.23, a_2.27, a_3.37,
-\bigl(b+\tfrac{c}{2}\bigr).47, \bigl(\tfrac{c}{2}-a_2\bigr).57,
\bigl(\tfrac{c}{2}-a_3\bigr).67, 0\Bigr).
  \end{equation*}

  For Case~\ref{item:omega}\ref{item:dim2} we need to start with a
different Lie algebra.  We take
  \begin{equation*} (a.47, -a.57, b.37, -a.17, a.27, -b.67, 0)
  \end{equation*} and the cocalibrated $\G_2$-structure $\varphi\in
\Lambda^3 \mfg^*$ given by the same formula as above.  Moreover, let
$\mfa = \spa{e_3,e_6}$, $\omega_0 = -(e^{12}+e^{45})\otimes e_3 -
(e^{15}+e^{24})\otimes e_6$ and $\nu \in \mfg^*\otimes \mfa$ defined
by $\nu(e_3) = -be_3+2a e_6$, $\nu(e_6) = -2a e_3+b e_6$.  Then
$f.(e^{12}+e^{45}) = 2a(e^{15}+e^{24})$ and $f.(e^{15}+e^{24}) =
-2a(e^{12}+e^{45})$, so the shear gives a cocalibrated
$\G_2$-structure on
  \begin{equation*} (a.47, -a.57, -2a.67+12+45, -a.17, a.27,
2a.37+15+24, 0).
  \end{equation*}
\end{example}

\subsubsection{Calibrated $\G_2$-structures}

Given a $\G_2$-structure~$\varphi$ on an almost Abelian Lie algebra
$\mfg$ and unit-length $\alpha\in \mfg^*$ in the annihilator of~$\mfu$, there is
also an almost Hermitian structure $(\sigma,\rho)$ on the codimension
one Abelian ideal~$\mfu$ related to $\varphi$ via $\varphi =
\sigma\wedge \alpha+\rho$, cf.~\cite{MC}.  The \emph{calibrated} case
is when $d\varphi = 0$.  In this situation $f = \ad(X)|_{\mfu}\in
\lie{sl}(\mfu,\rho) \coloneqq \{\,g\in \End(\mfu)\mid g.\rho =
0\,\}\cong \lie{sl}(3,\bC)$ and so $\tr(f) = 0$ and $[f,J] = 0$,
cf.~\cite{F2}.  We aim at a partial classification of left-invariant
shear data $(\inc,\id_{\mfa},\omega)$ as in
Proposition~\ref{pro:sheardata-AALAs} for which the shear of $\varphi$
is again calibrated.

When $U$ is a subspace of~$\mfu$, we write $\proj_U$ for the
orthogonal projection $\mfu \to U$.

\begin{proposition}
  \label{pro:shearcalG2} Let $(\mfg,\mfu,\varphi)$ be a
seven-dimensional almost Abelian Lie algebra with a calibrated
$\G_2$-structure. Fix an $f$-invariant subspace $\mfa$ of $\mfu$ and
let $\omega = \omega_0+\alpha\wedge \nu\in \Lambda^2 \mfg^*\otimes
\mfa$ be as in Proposition~\ref{pro:sheardata-AALAs}.
  \begin{enumerate}[\upshape(a),wide]
  \item\label{item:cal-dim2} If $\ker(\omega_0)$ is $J$-invariant and
of dimension~$2$, then $(\inc,\id_{\mfa},\omega)$ shears
$(\mfg,\varphi)$ to a calibrated structure if and only if $\dim(\mfa)
= 2$ and for some $Y\in \mfa\setminus\{0\}$ either
    \begin{enumerate}[\upshape(i),wide=3em]
    \item\label{item:fonanotequal0} $\omega_0 =
\tr(f|_{\mfa})\bigl(JY\hook \rho\otimes Y-Y\hook \rho\otimes JY\bigr)$
and
      \begin{equation*} \nu \in
-\tr(f|_{\mfa})\,\proj_{\mfa}+\{\,\hat{\nu}\in \sL(\mfu,\rho) \mid
\hat{\nu}|_{\mfa} = 0,\ \hat{\nu}(\mfu)\subset \mfa\,\},
      \end{equation*} or
    \item\label{item:fonaequal0} $f|_{\mfa} = 0$, $\omega_0 = (a
Y\hook \rho+b JY\hook \rho)\otimes Y+ (c Y\hook \rho-a JY\hook
\rho)\otimes JY$ for $(a,b,c)\in \bR^3\setminus\{\mathbf 0\}$ with
$a^2+bc = 0$ and
      \begin{equation*} \nu\in \tilde{\nu}+\{\,\hat{\nu}\in
\sL(\mfu,\rho) \mid \hat{\nu}|_{\mfa} = 0,\ \hat{\nu}(\mfu)\subset
\mfa\,\}
      \end{equation*} with $\tilde{\nu}(Y) = cY-aJY$, $\tilde{\nu}(JY)
= -aY-bJY$, $\tilde{\nu}(\mfa^{\perp}) = \{0\}$.
    \end{enumerate}
  \item\label{item:cal-dim4J} If $\dim(\mfa) = 4$ and $\mfa =
\ker(\omega_0)$ is $J$-invariant, then $(\inc,\id_{\mfa},\omega)$
shears $(\mfg,\varphi)$ to a calibrated structure if and only if
$\omega_0 = \tilde{\omega}\otimes Y$ for a $Y\in \mfa$ with $\norm{Y}
= 1$ and $\tilde{\omega} \in \Lambda^2 \mfu^*$ decomposable, and
    \begin{equation*} \nu = \tilde{\nu} + \{\,\hat{\nu}\in
\lie{sl}(\mfu,\rho) \mid \hat{\nu}(Y) = 0,\hat{\nu}(\mfu)\subset
\mfa\,\}
    \end{equation*} where $\tilde{\nu}$ is given by
    \begin{gather*} \tilde{\nu}(Y) = -\tr(f|_{\mfa})Y-f(Y),\quad
\tilde{\nu}(JY) = \mu Y-\lambda JY-J\proj_U(f(Y)),\\ \tilde{\nu}|_U =
\lambda\id_U +\mu J|_U,\quad \tilde{\nu}(W) =
-\frac{\tilde{\omega}(W,JW)}{2\norm{\rho(Y,W,\any)}^2}
\rho(Y,W,\any)^{\sharp}\ \forall\, W \in \mfa^\bot
    \end{gather*} with $U = \spa{Y,JY}^{\perp}\cap \mfa$, $\lambda =
\tr(f|_{\mfa})+g(f(Y),Y)$ and $\mu = g(f(Y),JY)$.
  \item\label{item:cal-dim4nJ} If $\dim(\mfa) = 4$ and $\mfa =
\ker(\omega_0)$ is not $J$-invariant, then $(\inc,\id_{\mfa},\omega)$
shears $(\mfg,\varphi)$ to a calibrated structure if and only if
$\omega_0 = \tilde{\omega}\otimes Y$ for a unit length $Y\in \mfa \cap
J\mfa$ with $Y\hook \rho|_{\Lambda^2\mfa} = 0$ and decomposable
$\tilde{\omega} \in \mu JY\hook \rho + \llbracket \Lambda^{1,1} \mfu^*
\rrbracket$, such that $\mu\in \bR$ is non-zero, $f(Y) = \lambda Y$
for $\lambda\in \bR$ fulfilling $4\lambda JY\hook \rho|_{\Lambda^2
U} = -J^*\tilde{\omega}|_{\Lambda^2 U}$ and
    \begin{equation*} \nu = \tilde{\nu}+ \{\,\hat\nu \in \End(\mfu)
\mid [\hat\nu,J] = 0,\ \hat{\nu}|_{\spa{Y,JY}} = 0,\ \hat\nu(\mfu)
\subset \spa{Y,JY}\,\}
    \end{equation*} for $\tilde{\nu} =
\diag(-2\lambda,-4\lambda,2\lambda,0)$ with respect to $\mfu =
\spa{Y}\oplus \spa{JY}\oplus U\oplus JU$, $U = \spa{Y,JY}^{\perp}\cap
\mfa$.
  \end{enumerate}
\end{proposition}

\begin{proof} \ref{item:cal-dim2} First observe that we always have
$\im\omega_0\subset \mfa\subset \ker\omega_0$.  Consider the case
$\dim(\mfa) = 2$.  Choose $Y_1\in \ker(\omega_0) = \mfa$ and set $Y_2
= JY_1$.  We may then write $\omega_0 = \sum_{i = 1}^2 \omega_i\otimes
Y_i$ for two-forms $\omega_1,\omega_2$, where $\omega_2 = 0$ in the
case $\dim(\im \omega_0) = 1$.  The second equation
in~\eqref{eq:sheareq2} reads
  \begin{equation}
    \label{eq:cal-d2-kappa} Y_2^b\wedge \omega_1-Y_1^b\wedge \omega_2
= -\kappa(\sigma\wedge \omega_0) = \nu.\rho.
  \end{equation} As $\nu(\mfu)\subset \mfa$ and $\rho(Z,JZ,\cdot)=0$ for all $Z\in \mfu$,
        inserting $Y_{3-i}$ into equation \eqref{eq:cal-d2-kappa} gives $\omega_i\in
\{\tilde{\omega}_1,\tilde{\omega}_2\}$
 for $\tilde\omega_i :=
Y_i \hook\rho$, $i=1,2$. Hence, we may write $\omega_i
= \sum_{j = 1}^2 a_{ij} \tilde{\omega}_j$ for some $A = (a_{ij})\in
M_2(\bR)$.  Equation~\eqref{eq:cal-d2-kappa} then gives $\nu(Y_j) =
\sum_{i = 1}^2 c_{ij} Y_i$ with $C = (c_{ij}) = A^T D$ for $D
= \begin{psmallmatrix} 0 & -1 \\ 1 & 0 \end{psmallmatrix}$.

Since $\tilde{\omega}_i\wedge \tilde{\omega}_j=\delta_{ij} \tilde{\omega}_1^2$ for all
$i,j=1,2$, the first equation in \eqref{eq:sheareq2} reads
  \begin{equation}
    \label{eq:cal-d2-rho} 0 = \kappa(\rho\wedge \omega_0) = \sum_{i =
1}^2 \tilde{\omega}_i\wedge \omega_i =
(a_{11}+a_{22})\tilde{\omega}_1^2.
  \end{equation} So $\tr A = 0$, giving $A\in \lie{sl}(2,\bR)$.

  If $\dim(\mfa) = 1$, then $\omega_2 = 0$ and $\omega_1 =
a_{11}\tilde{\omega}_1$.  But \eqref{eq:cal-d2-rho} gives $a_{11} = 0$
and so $\omega_0 = 0$, which is a contradiction.  Thus $\dim(\mfa)=2$.

  As $\mfa$ is $f$-invariant and $[f,J] = 0$, we have $f(Y_j) =
\sum_{i = 1}^2 b_{ij} Y_i$ with $B = (b_{ij}) =
  \begin{psmallmatrix} b_{11}&b_{12}\\ -b_{12}&b_{11}
  \end{psmallmatrix} \in M_1(\bC) \subset M_2(\bR)$, i.e. $b_{21}=-b_{12}$ and $b_{22}=b_{11}$.
        Since $f.\rho =0$, we also get $f.\tilde{\omega}_j = \sum_{i = 1}^2 b_{ij}
\tilde{\omega}_i$.  Now the first equation in \eqref{eq:sheareq1} is
equivalent to $AB^T = -(B+C)A = -BA-A^TDA = -BA-\det(A)D$, which is
  \begin{gather*} 2a_{11} b_{11} = -(a_{12}+a_{21})b_{12},\quad
2a_{11} b_{12} = (a_{12}+a_{21})b_{11},\\ \det(A) = 2(a_{12}
b_{11}-a_{11} b_{12}).
  \end{gather*} The first two equations are solved if and only if
$a_{12}+a_{21}= 0 = a_{11}$ or $B = 0$.  In the first case, we must
have $a_{12}\neq 0$, as $A\neq 0$, and the last equation implies
$a_{12} = 2b_{11} = \tr(f|_{\mfa})$.  In the second case, we have
$f|_{\mfa} = 0$ and the last equation gives us $a_{11}^2+a_{12}a_{21}
= -\det(A) = 0$.  This gives the two cases claimed.

  \ref{item:cal-dim4J}~\&~\ref{item:cal-dim4nJ} First note that
$\dim\ker\omega_0 = \dim\mfa = 4$, implies that $\dim\im\omega_0
\leqslant \dim\Lambda^2(\ker\omega_0{}^\bot) = 1 $.  We thus have
$\omega_0 = \tilde{\omega}\otimes Y$ for some non-zero decomposable
two-form $\tilde{\omega}$ and a $Y\in\mfa$ with $\norm{Y} = 1$. Now
the first equation in \eqref{eq:sheareq2} reads $Y\hook \rho \wedge
\tilde{\omega} = 0$ and the second equation in \eqref{eq:sheareq2}
reads $JY^b\wedge \tilde{\omega} = \nu.\rho$. Moreover, using
$f.\omega_0 = \tr(f|_{\mfa})\omega_0$ we see that the first equation
in \eqref{eq:sheareq1} is equivalent to $\nu(Y) = -f(Y)-\tr(f|_{\mfa})
Y$.

  \ref{item:cal-dim4J} Here, $Y\hook \rho\wedge \tilde{\omega}=Y\hook
(\rho\wedge \tilde{\omega}) = 0$ holds as $\mfa =
\ker(\tilde{\omega})$ is $J$-invariant, i.e.\ $\tilde{\omega}$ is a
$(1,1)$-form.  So we only have to check that the given $\tilde{\nu}$
is a solution of $JY^b\wedge \tilde{\omega} = \tilde{\nu}.\rho$.  To
simplify the notation, we set
  \begin{equation*} a(W) =
\frac{\tilde{\omega}(W,JW)}{2\norm{\rho(Y,W,\any)}^2} =
\frac{\tilde{\omega}(W,JW)}{2\norm{\rho(JY,JW,\any)}^2}
  \end{equation*} for $W\in \mfa^{\perp}$.  For such~$W$, we have
$JW\in \mfa^{\perp}$ too, so $\tilde{\nu}(W)\in U\subset \mfu$.
Moreover, as $\rho$ is zero when evaluated on any pair $A,JA$ and as $f.\rho=0$,
straightforward computations give us that $\tilde{\nu}.\rho$ is zero on $\Lambda^3 \mfa + \Lambda^2
\mfa\wedge \mfa^{\perp} + U\wedge \Lambda^2 \mfa^{\perp}$.  As
$\tilde{\nu}(JW) = -J\tilde{\nu}(W)$ for $W\in \mfa^{\perp}$, we
obtain $(\tilde{\nu}.\rho)(Z,W,JW) = 2\rho(Z,JW,\tilde{\nu}(W))$.
Thus
  \begin{equation*} (\tilde{\nu}.\rho)(Z,W,JW) =
2\rho(Z,JW,\tilde{\nu}(W)) = 2a(W)g(\rho(Z,JW,\any),\rho(JY,JW,\any))
  \end{equation*} for all $W\in\mfa^{\perp}$ and all $Z\in
\spa{Y,JY}$.  For $Z = Y$, this equals zero since $\rho(Y,JW,\any) =
J\rho(JY,JW,\any)$; for $Z = JY$, we obtain
  \begin{equation*} (\tilde{\nu}.\rho)(JY,W,JW) = \tilde{\omega}(W,JW)
= (JY^b \wedge \tilde{\omega})(JY,W,JW),
  \end{equation*} as required.

  \ref{item:cal-dim4nJ} Inserting $JY$ into $Y\hook \rho\wedge
\tilde{\omega} = 0$, we get $JY\hook \tilde{\omega} = 0$, as $Y\hook
\rho$ has rank four.  So $JY\in \ker(\omega_0) = \mfa$ and the
equation may be considered as one on the four-dimensional space
$\spa{Y,JY}^{\perp}$.  But then $\tilde{\omega}$ solves
$Y\hook\rho\wedge\tilde\omega = 0$ if and only if $\tilde{\omega}\in
\llbracket \Lambda^{1,1} (\spa{Y,JY}^{\perp})^*\rrbracket\oplus
\spa{JY\hook \rho}$.  As $\mfa = \ker(\tilde{\omega})$ is not
$J$-invariant, the $\spa{JY\hook \rho}$-part of $\tilde{\omega}$ must
be non-zero as claimed.

  Assume now that $JY^b\wedge \tilde{\omega} = \nu.\rho$ holds.
Inserting both $Y$ and $JY$ into this equation one obtains
$\rho(f(Y),JY,\any)-\rho(Y,\nu(JY),\any) = 0$ and so
$\rho(Y,\nu(JY)+Jf(Y),\any) = 0$.  Thus $\nu(JY) = -Jf(Y)$ up to terms
in $\spa{Y,JY}$.  However, the hypotheses give $JU\cap \mfa = \{0\}$,
so $f(Y),\nu(JY)\in \spa{Y,JY}$.

  Inserting $Z\in U$ and $JZ\notin\mfa$ in to the same equation, we
get $\nu(JZ) = J\nu(Z)$ up to terms in $\spa{Z,JZ}$.  So there are
$\lambda_1,\lambda_2\in \bR$ such that $\nu(Z)-\lambda_1 Z,
\nu(JZ)-\lambda_2 Z\in \spa{Y,JY}$ for all $Z\in U$. Take now $Z,W\in
U$ such that $Y,Z,W$ is a $\bC$-basis of~$\mfu$. Then $\rho(Y,Z,W) =
0$ from $Y\hook \rho\wedge \tilde{\omega} = 0$. So we have $Y\hook
\rho|_{\Lambda^2\mfa}=0$ and must have $\rho(JY,Z,W)\neq 0$. Hence,
  \begin{equation*} 0 = (JY^b\wedge \tilde{\omega})(Y,Z,W) =
(\nu.\rho)(Y,Z,W) = -\rho(\nu(Y),Z,W)
  \end{equation*} implies $\nu(Y)\in \spa{Y}$.  So $f(Y) = \lambda Y$
for some $\lambda\in \bR$.  As
  \begin{equation*} 0 = (f.\rho)(JY,Z,W) =
-(\tr(f|_{\mfa})-\lambda)\rho(JY,Z,W),
  \end{equation*} we must have $\tr(f|_{\mfa}) = \lambda$.  Thus
$\nu(Y) = -2\lambda Y$.

  Similarly, $(\nu.\rho)(Y,Z,JW) = 0$ yields $\lambda_1 = 2\lambda$
and $\nu.\rho(JY,Z,W)=0$ implies that $\nu(JY) = a Y-4\lambda JY$ for
some $a\in \bR$.  The equality $0 = (\nu.\rho)(Y,JZ,JW)$ gives us $\lambda_2 =
0$ and then the equality $0 = (\nu.\rho)(JY,Z,JW)$ gives us $a=0$.

  Finally, inserting $JY,JZ,JW$ into $JY^b\wedge \tilde{\omega} =
\nu.\rho$, we get $4\lambda\rho(JY,Z,W) = -\tilde{\omega}(JZ,JW)$.
This gives that the difference $\hat{\nu}$ between two solutions
of $JY^b\wedge \tilde{\omega} = \nu.\rho$ are those
$\hat{\nu}\in \End(\mfu)$ with $\hat{\nu}|_{\spa{Y,JY}} = 0$,
$\hat{\nu}(\mfu)\subset \spa{Y,JY}$ and $\hat{\nu}\in \sL(\mfu,\rho)$,
which here is equivalent to $[\hat{\nu},J] = 0$, as claimed. Conversely, the above computations
show that $\tilde{\nu}$ as in the statement fulfils
$JY^b\wedge \tilde{\omega} = \tilde{\nu}.\rho$, completing the proof.
\end{proof}

Let us give explicit examples for all the cases in
Proposition~\ref{pro:shearcalG2}.
\begin{example}
  \label{ex:shearcalibrated} Start with the almost Abelian Lie algebra
defined by
  \begin{equation*} (a.17, a.27, b.37,b.47,c57,c.67, 0)
  \end{equation*} for $a, b, c\in \bR$ with $a+b+c =0$ and consider
the calibrated $\G_2$-structure $\varphi=
127+347+567+135-146-236-245$.

  Case~\ref{item:cal-dim2}\ref{item:fonanotequal0}: Taking $\mfa =
\spa{e_1,e_2}$, $\omega_0 = 2a\bigl((e^{36}+e^{45})\otimes
e_1+(e^{35}-e^{46})\otimes e_2\bigr)$ and $\nu\in \End(\mfu)$ defined
by $\nu|_{\mfa} = -2a\id_{\mfa}$, $\nu(e_i) = 0$ for $i = 3,4,5,6$, we
may shear $\varphi$ to a calibrated $\G_2$-structure on
  \begin{equation*} \qquad (-a.17-2a.(36+45), -a.27-2a.(35-46),
b.37,b.47,c.57,c.67, 0).
  \end{equation*}

  Case~\ref{item:cal-dim2}\ref{item:fonaequal0}: We assume now that
$a=0$ and take $\mfa=\spa{e_1,e_2}$, $a_1,a_2,a_3\in \bR$ with
$a_1^2+a_2a_3=0$,
  \begin{equation*} \omega_0 = \bigl(a_1 (e^{35}-e^{46}) + a_2
(e^{36}+e^{45})\bigr)\otimes e_1 - \bigl(a_3 (e^{35}-e^{46}) - a_1
(e^{36}+e^{45})\bigr)\otimes e_2
  \end{equation*} and $\nu\in \mfu^*\otimes \mfa$ defined by $\nu(e_1)
= a_3 e_1+a_1 e_2$, $\nu(e_2) = a_1 e_1-a_2 e_2$ and $\nu(e_i) = 0$
for all $i=3,\dots,6$.  With this data, we may shear $\varphi$ to a
calibrated $\G_2$-structure on
  \begin{multline*} (a_3.17-a_1.(35-46-27)- a_2.(36+45),\\
-a_2.27+a_3.(35-46)-a_1.(36+45-17), b.37, b.47, -b.57, -b.67, 0).
  \end{multline*}

  Case~\ref{item:cal-dim4J}: Taking $\mfa = \spa{e_1,e_2,e_3,e_4}$,
$\omega_0 = -e^{56}\otimes e_1$ and $\nu\in \End(\mfu)$ given by
$\nu|_{\spa{e_1,e_2}} = -(3a+2b) \id_{\spa{e_1,e_2}}$,
$\nu|_{\spa{e_3,e_4}} = (3a+2b) \id_{\spa{e_3,e_4}}$, $\nu(e_5) =
\tfrac12e_3$ and $\nu(e_6) = -\tfrac12e_4$, we may shear $\varphi$ to
a calibrated $\G_2$-structure on
  \begin{equation*} (2c.17+56,2c.27,-3c.37+\tfrac{1}{2}.57,
-3c.47-\tfrac{1}{2}.67,c.57,c.67, 0).
  \end{equation*}

  Case~\ref{item:cal-dim4nJ}: Taking $\mfa = \spa{e_1,e_2,e_4,e_5}$,
$\omega_0 = 4a e^{36}\otimes e_1$ and $\nu\in \End(\mfu)$ with
$\nu(e_1) = -2a e_1$, $\nu(e_2) = -4a e_2$, $\nu|_{\spa{e_4,e_5}} =
2a\id_{\spa{e_4,e_5}}$ and $\nu|_{\spa{e_3,e_6}} = 0$, we may shear
$\varphi$ to a calibrated $\G_2$-structure on
  \begin{equation*} (-a.17-4a.36,-3a.27,b.37,(b+2a).47,(a-b).57,-(a+b).67,
0).
  \end{equation*}
\end{example}

Note that in cases \ref{item:cal-dim4J} and~\ref{item:cal-dim4nJ} of
Proposition~\ref{pro:shearcalG2}, the shear Lie algebra is of the form
$(\mfh_3\oplus \bR^3) \rtimes \bR$, where $\mfh_3$ is the
three-dimensional Heisenberg algebra. We thus obtain calibrated
$\G_2$-structures on such Lie algebras as the shears of calibrated
$\G_2$-structures on almost Abelian Lie algebras. In fact, we get all
possible calibrated $\G_2$-structures on that class of Lie algebras this way:

\begin{corollary}
  \label{co:calG2} Let $\mfg$ be a seven-dimensional Lie algebra with
a codimension one nilpotent ideal $\mfu\cong \mfh_3\oplus \bR^3$.  Let
$\varphi\in \Lambda^3 \mfg^*$ be a $\G_2$-structure.  Fix $X \bot
\mfu$ with $\norm X = 1$.

  Write $h = \ad(X)|_{\mfu}\in \lie{der}(\mfu)$ and
$(g,J,\sigma,\rho)$ for the special almost Hermitian structure
on~$\mfu$ induced by~$\varphi$.  Set
  \begin{equation*} U_1 = [\mfu,\mfu]\oplus J[\mfu,\mfu],\quad U_2 =
U_1^{\perp}\cap \mfz(\mfu).
  \end{equation*}

  Then $\varphi\in \Lambda^3 \mfg^*$ is calibrated if and only if
$J[\mfu,\mfu]\subset \mfz(\mfu)$ and either
  \begin{enumerate}[\upshape (i)]
  \item\label{item:cor-cal-1} $\mfz(\mfu)$ is $J$-invariant and there
are $\lambda,\mu\in \bR$ and linear maps $h_{ij} \colon U_j\to U_i$
with $[h_{ij},J] = 0$ such that
    \begin{equation*} h = \begin{pmatrix} -2 \lambda & h_{12} & h_{13}
\\ 0 & 3\lambda + \mu J & h_{23}+h \\ 0 & 0 & -\lambda-\mu J
      \end{pmatrix}
    \end{equation*} on $\mfu = U_1\oplus U_2\oplus U_3$, where $U_3:=\mathfrak{z}(\mfu)^{\perp}$ and $h
\colon U_3\to U_2$ is given by
    \begin{equation*} h(Z) =
-\frac{\norm{[Z,JZ]}^2}{2\norm{\rho([Z,JZ],Z,\any)}^2}
\rho([Z,JZ],Z,\any)^{\sharp},
    \end{equation*} for all $Z \in U_3$, or
  \item\label{item:co-cal-2} $\mfz(\mfu)$ is not $J$-invariant,
$\rho$~is zero on~$[\mfu,\mfu]\wedge \Lambda^2 \mfz(\mfu)$ and there
is an $h_1\in \lie{sl}(\mfu,\rho)$ with $U_1 \subset \ker h_1$ and
$h_1(\mfz(\mfu))\subset \mfz(\mfu)$ such that
    \begin{equation*} h =
\diag(-2\lambda,-6\lambda,3\lambda,-\lambda)+h_1
    \end{equation*} on $\mfu = [\mfu,\mfu]\oplus J[\mfu,\mfu]\oplus
U_2\oplus JU_2$ with $\lambda\in \bR$ specified by $-8\lambda
\rho(Z_1,Z_2,\any)^{\sharp} = J[JZ_1,JZ_2]$ for any basis $Z_1,Z_2$
of~$U_2$.
  \end{enumerate}
\end{corollary}

\begin{proof} The derivations $h$ of $\mfu$ on the shear obtained from
Proposition~\ref{pro:shearcalG2} \ref{item:cal-dim4J} and
\ref{item:cal-dim4nJ} are exactly those given in
Corollary~\ref{co:calG2} \ref{item:cor-cal-1} and \ref{item:co-cal-2}:
this may be seen by straightforward computations using $h =
f+\tilde{\nu}+\hat{\nu}$ and $[W_1,W_2] = \tilde{\omega}(W_1,W_2)Y$,
for any $W_1,W_2\in \mfz(\mfu)^{\perp}$, and for~\ref{item:co-cal-2},
$\rho(Z_1,Z_2,\any)^{\sharp}\in \spa{JY}$ in
Proposition~\ref{pro:shearcalG2}\ref{item:cal-dim4nJ}.  So the
direction ``$\Leftarrow$'' follows.

For the converse direction, we show that we can shear, with
left-invariant data $(\inc,\id_{\mfa},\omega_0)$, any calibrated
$\G_2$-structure $\varphi$ on an almost nilpotent Lie algebra of the
form $(\mfh_3\oplus \bR^3)\rtimes \bR$ to one on an almost Abelian Lie
algebra. As a left-invariant shear can be inverted by Theorem~\ref{th:duality}, we may
obtain $\varphi$ as the shear of a calibrated $\G_2$-structure on an
almost Abelian Lie algebra. Now
Proposition~\ref{pro:shearcalG2} \ref{item:cal-dim4J} and
\ref{item:cal-dim4nJ} contain all possible calibrated shears of
calibrated almost Abelian Lie algebras to Lie algebras of the form $(\mfh_3\oplus \bR^3)\rtimes \bR$
provided we have $\mfa=\ker(\omega_0)$. However, if $\mfa\subset \ker(\omega_0)$, we may simply enlarge $\mfa$
to the $f$-invariant subspace $\ker(\omega_0)$ and the direction ``$\Rightarrow$'' then follows.

So let $\varphi$ be calibrated and note that $\mfz(\mfu)$ is an
$h$-invariant subspace of~$\mfu$.  For the shear, we take $\mfa =
\mfz(\mfu)$ and $\omega = \omega_0+\alpha\wedge \nu$ with $\omega_0\in
\Lambda^2\mfu^*\otimes \mfa$, $\nu\in \mfu^*\otimes
\mfa\subset\End(\mfu)$ and $\alpha\in \mfg^*$ uniquely defined by
$\alpha(X) = 1$ and $\alpha(\mfu) = \{0\}$. To shear to an almost
Abelian Lie algebra, we must take $\omega_0 =
-[\proj_{\mfu}(\any),\proj_{\mfu}(\any)]$. Then $\im(\omega_0) =
[\mfu,\mfu]$. Similarly to the proof of
Proposition~\ref{pro:sheardata-AALAs} and using that $\mfa$ is
central, we get $\gamma = \alpha\otimes h|_{\mfa}$ and $\eta =
\gamma-\omega|_{\mfa\otimes \mfg} = \alpha\otimes (h+\nu)|_{\mfa}$.
Furthermore, the Jacobi identity gives us $d\omega_0|_{\Lambda^3 \mfu}
= 0$, so $d\omega_0 = \alpha\wedge h.\omega_0$. Hence,
$(\inc,\id_{\mfa},\omega)$ defines left-invariant shear data on~$G$ if
and only if
  \begin{equation*}
    \begin{split} 0 &= d\omega+\eta\wedge\omega = \alpha\wedge
h.\omega_0 - \alpha\wedge d_{\mfu}\nu + \alpha\wedge (h+\nu)\circ
\omega_0\\ &= \alpha\wedge (h.\omega_0-d_{\mfu}\nu+(h+\nu)\circ
\omega_0),
    \end{split}
  \end{equation*} giving $h.\omega_0-d_{\mfu}\nu+(h+\nu)\circ \omega_0
= 0$, where $d_{\mfu}$ is the differential of~$\mfu$.  However, $h$ is
a derivation, so
  \begin{equation*} (h.\omega_0-d_{\mfu}\nu)(Z_1,Z_2) =
[h(Z_1),Z_2]+[Z_1,h(Z_2)]+\nu[Z_1,Z_2] = -\bigl((h+\nu)\circ \omega_0\bigr)(Z_1,Z_2)
  \end{equation*} for all $Z_1, Z_2\in \mfu$.  Thus
$(\inc,\id_{\mfa},\omega)$ always defines left-invariant shear data on
$G$.

  \ref{item:cor-cal-1} Here, we set $\nu|_{\mfa} = 0$ and $\nu(Z) =
\tfrac{\norm{[Z,JZ]}^2}
{2\norm{\rho([Z,JZ],Z,\any)}^2}\rho([Z,JZ],Z,\any)^{\sharp}$ for any
$Z\in \mfa^{\perp}\subset \mfu$. The shear is again calibrated if and
only if \eqref{eq:sheareq2} holds. The first equation in
\eqref{eq:sheareq2} is equivalent to the vanishing of the
anti-symmetrisation of $\rho([\any,\any],\any,\any)$, and so is
satisfied since $[\mfu,\mfu]\subset \mfz(\mfu)$ and $\mfz(\mfu)$ is
four-dimensional and $J$-invariant.  The second equation in
\eqref{eq:sheareq2} is given by $\gamma = \nu.\rho$ with
$\gamma(X,Y,Z) = \sum_{\mathrm{cyclic}}g([X,Y],JZ)$ for $X,Y,Z\in
\mfu$.  As $\nu(\mfu)\subset \mfz(\mfu)\cap ([\mfu,\mfu]\oplus
J[\mfu,\mfu])^{\perp}$, both sides of the equation are zero on
$\Lambda^2 \mfz(\mfu)\wedge \mfu+\Lambda^3 ([\mfu,\mfu]\oplus
J[\mfu,\mfu])^{\perp}$.  Finally, a straightforward computation yields
  \begin{equation*} (\nu.\rho)(Y,Z,JZ) =
\frac{\norm{[Z,JZ]}^2}{\norm{\rho([Z,JZ],Z,\any)}^2}
g(\rho(JY,Z,\any),\rho([Z,JZ],Z,\any))
  \end{equation*} for any $Y\in [\mfu,\mfu]\oplus J[\mfu,\mfu]$ and
any $Z\in \mfz(\mfu)^{\perp}$.  For $Y\in [\mfu,\mfu]$, the right-hand
side is zero and for $Y\in J[\mfu,\mfu]$, one has $JY\in \spa{[Z,JZ]}$
and so $\nu.\rho(Y,Z,JZ) = g([Z,JZ],JY) = \gamma(Y,Z,JZ)$ as we
wanted.  Hence, the shear is calibrated.

  \ref{item:co-cal-2} Inserting into $0 = d\varphi$ two non-zero
elements of $\mfz(\mfu)^{\perp} \subset \mfu$ and two elements of
$\mfz(\mfu)$, we obtain $\rho|_{[\mfu,\mfu]\wedge \Lambda^2
\mfz(\mfu)} = 0$.  This implies $J[\mfu,\mfu]\subset \mfz(\mfu)$ as
otherwise we may take $Y\in [\mfu,\mfu]$, $Z_1\in \mfz(\mfu)\cap J\mfz(\mfu)$ and $Z_2\in
\mfz(\mfu)$ such that $Y,Z_1,Z_2$ is a $\bC$-basis of $\mfu$ and so
must have $0\neq \rho(Y,Z_1,Z_2)$ or $0\neq \rho(Y,JZ_1,Z_2)$, a
contradiction.  After these preliminary considerations,
set $U:=\mfz(\mfu)\cap ([\mfu,\mfu]\oplus J[\mfu,\mfu])^{\perp}$, define
$\lambda\in \bR$ via the formula $4\lambda \rho(Z_1,Z_2,\any)^{\sharp}
= -J[Z_1,Z_2]$, where $Z_1,Z_2$ is any basis of $JU$,
and define $\nu\in \mfu^*\otimes \mfa$ by $\nu(Y) = -2\lambda Y$,
$\nu(JY) = -4\lambda JY$ for all $Y\in [\mfu,\mfu]$, $\nu(Z) =
2\lambda Z$ for all $Z\in U$ and $\nu|_{JU} = 0$.  Firstly,
$\rho|_{[\mfu,\mfu]\wedge \Lambda^2 \mfz(\mfu)} = 0$ implies that the
anti-symmetrisation of $\rho([\any,\any],\any,\any)$ vanishes.
Furthermore, both $\nu.\rho$ and $\gamma$ as above, are zero on
$\Lambda^3 \mfz(\mfu)+\left(JY^{\perp}\cap \mfz(\mfu)\right)\wedge
\Lambda^2 \mfu$.  Finally, for $Y\in [\mfu,\mfu]$ and $Z_1,Z_2\in JU$ we get
  \begin{equation*}
    \begin{split} (\nu.\rho)(JY,Z_1,Z_2) &= 4\lambda\rho(JY,Z_1,Z_2) =
g(4\lambda \rho(Z_1,Z_2,\any)^{\sharp},JY)\\ &= -g(J[Z_1,Z_2],JY) =
g([Z_1,Z_2],J(JY)) = \gamma(JY,Z_1,Z_2),
    \end{split}
  \end{equation*} as required.
\end{proof}

\subsubsection{Almost semi-K\"ahler structures}

An almost Hermitian structure $(g,J,\sigma)$ on a $2n$-dimensional
manifold is called \emph{almost semi-K\"ahler} if $d(\sigma^{n-1}) =
0$.  Suppose $(g,J,\sigma)$ is an almost semi-K\"ahler structure on a
$2n$-dimensional almost Abelian Lie algebra $(\mfg,\mfu)$. Fix a unit
vector $X\in \mfu^\bot \subset \mfg$ and let $\alpha\in \mfu^\circ
\subset \mfg^*$ be the element with $\alpha(X) = 1$.  Then $\sigma =
(JX)^b\wedge\alpha+\sigma_1$ for some $\sigma_1$ with kernel
$\spa{X,JX}$.  Thus
\begin{equation*} \sigma^{n-1} = (n-1) (JX)^b\wedge
\sigma_1^{n-2}\wedge \alpha + \sigma_1^{n-1}
\end{equation*} and the almost semi-K\"ahler condition is equivalent
to $f.\sigma_1^{n-1} = 0$ for $f = \ad(X)|_{\mfu}$.  Since
$\sigma_1^{n-1}$ defines a volume form on $U \coloneqq
\spa{X,JX}^{\perp}$, this is the same as $f(JX) = \tr(f) JX$.

The first equation in \eqref{eq:sheareq2} is always satisfied, as
$\sigma_1^{n-1}\wedge \omega_0$ is an $n$-form with values in $\mfa$
on the $(n-1)$-dimensional vector space~$\mfu$.

Let us consider the case $\im(\omega_0) = \spa{JX}\subset
\ker(\omega_0) = \mfa$ with $\mfa$ an $f$-invariant subspace
of~$\mfu$.  Then $\omega_0 = \tilde{\omega}\otimes JX$ for a non-zero
$\tilde{\omega}\in \Lambda^2 \mfu^*$ and the first equation in
\eqref{eq:sheareq1} is equivalent to $f.\tilde{\omega} = \lambda
\tilde{\omega}$ and $\nu(JX) = -(\lambda+\tr(f)) JX$, for some
$\lambda\in \bR$.  The second equation in \eqref{eq:sheareq2} reads
$\tilde{\omega}\wedge\sigma_1^{n-2} = \nu.\sigma_1\wedge
\sigma_1^{n-2}$.  This is fulfilled if and only if
$\tilde{\omega}-\nu.\sigma_1\in [\Lambda_0^{1,1}U^*]\oplus \llbracket
\Lambda^{2,0}U^*\rrbracket$. Hence, if, e.g., $\tilde{\omega}\in [\Lambda_0^{1,1}U^*]\oplus \llbracket
\Lambda^{2,0} U^*\rrbracket$ is such that $f.\tilde{\omega} = \lambda
\tilde{\omega}$ for some $\lambda\in \bR$, then one obtains left-invariant shear data with the
shear being again almost semi-K\"ahler by taking
$\nu\in \End(\mfu)$ with $\nu|_U = 0$ and $\nu(JX) = -(\lambda+\tr(f))
JX$.

\begin{example}\label{ex:almostSemiKaehler} To get an explicit
example, take the Lie algebra
  \begin{equation*} (a_1.16,a_2.26,a_3.36,a_4.46,a_5.56,0),\qquad
\sum_{i=1}^4 a_i = 0.
  \end{equation*} Then $\mfg$ admits an almost semi-K\"ahler structure
with $\sigma = 12+34+56$.  One choice of shear is via $\omega_0 =
-e^{13}\otimes e_5$, $\nu = (a_1+a_3-a_5)e^5\otimes e_5$, giving an
almost semi-K\"ahler structure on
  \begin{equation*}
(a_1.16,a_2.26,a_3.36,a_4.46,(a_1+a_3).56+13,0).
  \end{equation*} For other choices, take $\tilde{\omega}$ to be
$e^{14}$, $e^{23}$ or $e^{24}$.

Another example may be obtained if $a1=a2=-a3$. Then $(\mfg,\sigma)$
is even \emph{semi-K\"ahler}, i.e.\ $J$ is integrable.
Moreover, if we shear $(\mfg,\sigma)$ with
$\omega_0 = -(e^{13}+e^{24})\otimes e_5$, $\nu = -a_5 e^5\otimes e_5$,
we get a semi-K\"ahler structure on
\begin{equation*}
  (a_1.16,a_1.26,-a_1.36,-a_1.46,13+24,0)
\end{equation*}
by the above and Proposition~\ref{pro:Nijenhuis} as $e^{13}+e^{24}$ is
a $(1,1)$-form on $(\mfg,J)$.
\end{example}

\providecommand{\bysame}{\leavevmode\hbox to3em{\hrulefill}\thinspace}
\providecommand{\MR}{\relax\ifhmode\unskip\space\fi MR }
\providecommand{\MRhref}[2]{%
  \href{http://www.ams.org/mathscinet-getitem?mr=#1}{#2}
}
\providecommand{\href}[2]{#2}


\begin{thebibliography}{ACD01}

\bibitem[ACD01]{ACD}
D.~V. Alekseevsky, V.~Cort\'{e}s, and C.~Devchand, \emph{Special complex
  manifolds}, J. Geom. Phys. \textbf{42} (2001), no.~1--2, 85--105.

\bibitem[BDV09]{BDV}
M.~L. Barberis, I.~G. Dotti, and M.~Verbitsky, \emph{Canonical bundles of
  complex nilmanifolds, with applications to hypercomplex geometry}, Math. Res.
  Lett. \textbf{16} (2009), no.~2, 331--347.

\bibitem[Bry06]{Br}
R.~L. Bryant, \emph{Some remarks on $\mathrm{G}_2$-structures}, Proceedings of
  the {G}\"okova {G}eometry-{T}opology {C}onference 2005 (S.~Akbulut,
  T.~\"Onder, and R.~J. Stern, eds.), G\"okova Geometry-Topology Conferences,
  International Press of Boston, Inc., 2006, pp.~75--109.

\bibitem[CF11]{CF}
D.~Conti and M.~Fern\'{a}ndez, \emph{{N}ilmanifolds with a calibrated
  $\mathrm{G}_2$-structure}, Differential Geom. Appl. \textbf{29} (2011),
  no.~4, 493--506.

\bibitem[CFS11]{CFS}
D.~Conti, M.~Fern\'{a}ndez, and J.~A. Santisteban, \emph{{S}olvable lie
  algebras are not that hypo}, Transform. Groups \textbf{16} (2011), no.~1,
  51--69.

\bibitem[CI07]{CI}
R.~Cleyton and S.~Ivanov, \emph{On the geometry of closed
  $\mathrm{G}_2$-structures.}, Commun. Math. Phys. 270, No. 1, 53--67 (2007).
  \textbf{270} (2007), no.~1, 53--67.

\bibitem[Fre12]{F1}
M.~Freibert, \emph{Cocalibrated structures on {L}ie algebras with a codimension
  one {A}belian ideal}, Ann. Global Anal. Geom. \textbf{42} (2012), no.~4,
  537--563.

\bibitem[Fre13]{F2}
\bysame, \emph{Calibrated and parallel structures on almost {A}belian {L}ie
  algebras}, preprint \url{arxiv:1307.2542 [math.DG]}, July 2013.

\bibitem[FS16]{FS}
M.~Freibert and A.~Swann, \emph{Solvable groups and a shear construction}, J.
  Geom. Phys. \textbf{106} (2016), 268--274.

\bibitem[FU13]{FU}
A~Fino and L.~Ugarte, \emph{On generalized {G}auduchon metrics}, Proc. Edinb.
  Math. Soc. (2) \textbf{56} (2013), no.~3, 733--753.

\bibitem[FV15]{FV}
A.~Fino and L.~Vezzoni, \emph{{S}pecial {H}ermitian metrics on compact
  solvmanifolds}, J. Geom. Phys. \textbf{91} (2015), 40--53.

\bibitem[GPS97]{GPS}
G.~W. Gibbons, G.~Papadopoulos, and K.~S. Stelle, \emph{{H}{K}{T} and {O}{K}{T}
  geometries on solition black hole moduli spaces}, Nuclear Phys. B
  \textbf{508} (1997), no.~3, 623--658.

\bibitem[Hit01]{Hi}
N.~Hitchin, \emph{Stable forms and special metrics}, Global Differential
  Geometry: The Mathematical Legacy of Alfred Gray (M.~Fernandez and J.~A.
  Wolf, eds.), Contemporary Mathematics, vol. 288, American Mathematical
  Society, 2001, pp.~70--89.

\bibitem[IP13]{IP}
S.~Ivanov and G.~Papadopoulos, \emph{{V}anishing theorems on $(l|k)$-strong
  {K}\"ahler manifolds with torsion}, Adv. Math. \textbf{237} (2013), 147--164.

\bibitem[Mac87]{McK2}
Kirill C.~H. Mackenzie, \emph{Lie groupoids and {L}ie algebroids in
  differential geometry}, London Mathematical Society Lecture Note Series, vol.
  124, Cambridge University Press, Cambridge, 1987.

\bibitem[Mac05]{McK1}
\bysame, \emph{General theory of {L}ie groupoids and {L}ie algebroids}, London
  Mathematical Society Lecture Note Series, vol. 213, Cambridge University
  Press, Cambridge, 2005.

\bibitem[MC06]{MC}
F.~Mart\'\i n~Cabrera, \emph{$\mathrm{SU}(3)$-structures on hypersurfaces of
  manifolds with $\mathrm{G}_2$-structures}, Monatsh. Math. \textbf{148}
  (2006), no.~1, 29--50.

\bibitem[MS15]{MS}
\'{O}. Maci\'{a} and A.~Swann, \emph{Twist geometry of the c-map}, Comm. Math.
  Phys. \textbf{336} (2015), no.~3, 1329--1357.

\bibitem[Swa07]{Sw1}
A.~F. Swann, \emph{T is for twist}, Proceedings of the XV International
  Workshop on Geometry and Physics, Puerto de la Cruz, September 11--16, 2006
  (D.~Iglesias~Ponte et~al., ed.), Publicaciones de la Real Sociedad
  Matem\'{a}tica Espa\~{n}ola, vol.~11, Spanish Royal Mathematical Society,
  2007, pp.~83--94.

\bibitem[Swa10]{Sw2}
\bysame, \emph{Twisting {H}ermitian and hypercomplex geometries}, Duke Math. J.
  \textbf{155} (2010), no.~2, 403--431.

\bibitem[Swa16]{Sw3}
\bysame, \emph{Twists versus modifications}, Adv. Math. \textbf{303} (2016),
  611--637.

\bibitem[SYZ96]{SYZ}
A.~Strominger, S.-T. Yau, and E.~Zaslow, \emph{Mirror symmetry is t-duality},
  Nuclear Phys. B \textbf{479} (1996), no.~1--2, 243--259.

\bibitem[Uga07]{U}
L.~Ugarte, \emph{Hermitian structures on six-dimensional nilmanifolds},
  Transform. Groups \textbf{12} (2007), no.~1, 175--202.

\end{thebibliography}
\end{document}